\documentclass[12pt]{article}
\usepackage{amsmath}
\usepackage{amsthm}
\usepackage{amsfonts}
\usepackage{amssymb}
\usepackage{esint}
\usepackage{eucal}
\usepackage{mathrsfs}
\usepackage{graphicx,graphics}
\numberwithin{equation}{section}
\usepackage{graphicx,graphics}
\usepackage{pdfsync}
\usepackage{multirow}
\usepackage{endnotes}
\usepackage{color}
\def \dis {\displaystyle}
\def \tex {\textstyle}

\def \limi {\;\mathop{\longrightarrow}_{n\to\infty}\;}

\def \into {\int_\Omega}
\def \confai {-\kern -.5em\rightharpoonup}
\def \cqfd {\hfill$\Box$}
\def \div{\mbox{\rm div}}
\def \Div{\mbox{\rm Div}}
\def \curl{\mbox{\rm curl}}
\def \Curl{\mbox{\rm Curl}}
\def \tr {\mbox{\rm tr}}
\newcommand{\bfm}[1]{\mbox{\boldmath ${#1}$}}
\def \al {\alpha}

\def \Ga {\Gamma}
\def \de {\delta}

\def \ep {\varepsilon}
\def \om {\omega}
\def \Om {\Omega}
\def \la {\lambda}
\def \La {\Lambda}
\def \ph {\varphi}

\def \si {\sigma}

\def \AA {\mathbb A}
\def \BB {\mathbb B}

\def \NN {\mathbb N}

\def \RR {\mathbb R}

\def \D {\mathscr{D}}
\def \E {\mathscr{E}}

\def \H {\mathscr{H}}

\def \M {\mathscr{M}}

\def \beq {\begin{equation}}
\def \eeq {\end{equation}}
\def \ba {\begin{array}}
\def \ea {\end{array}}
\def \bs {\bigskip}
\def \ms {\medskip}
\def \ss {\smallskip}
\def \ecart {\noalign{\medskip}}
\newtheorem{Thm}{Theorem}[section]

\newtheorem{Pro}[Thm]{Proposition}
\newtheorem{Lem}[Thm]{Lemma}

\newtheorem{Adef}[Thm]{Definition}
\newenvironment{Def}{\begin{Adef}\rm}{\end{Adef}}
\newtheorem{Arem}[Thm]{Remark}
\newenvironment{Rem}{\begin{Arem}\rm}{\end{Arem}}
\newtheorem{Aexa}[Thm]{Example}
\newenvironment{Exa}{\begin{Aexa}\rm}{\end{Aexa}}
\newtheorem{Anot}[Thm]{Notation}

\def \refe #1.{\eqref{#1})}
\def \reff #1.{figure~\ref{#1}}
\def \refs #1.{Section~\ref{#1}}
\def \refss #1.{Subsection~\ref{#1}}
\def \refD #1.{Definition~\ref{#1}}
\def \refT #1.{Theorem~\ref{#1}}
\def \refL #1.{Lemma~\ref{#1}}
\def \refC #1.{Corollary~\ref{#1}}
\def \refP #1.{Proposition~\ref{#1}}
\def \refPt #1.{Properties~\ref{#1}}
\def \refR #1.{Remark~\ref{#1}}
\def \refE #1.{Example~\ref{#1}}
\def \refN #1.{Notation~\ref{#1}}
\newcounter{marnote}

\title{A new div-curl result. Applications to the homogenization of elliptic systems and to the weak continuity of the Jacobian}

\author{
\footnotesize
\begin{tabular}{cccc}
\normalsize\sc M. Briane & \normalsize\sc J. Casado D\'iaz
\\
Institut de Recherche Math\'ematique de Rennes & Dpto. de Ecuaciones Diferenciales y An\'alisis Num\'erico
\\
INSA de Rennes  & Universidad de Sevilla
\\
mbriane@insa-rennes.fr & jcasadod@us.es
\end{tabular}
}
\begin{document}
\maketitle
\begin{abstract}
In this paper a new div-curl result is established in an open set $\Om$ of $\RR^N$, $N\geq 2$, for the product of two sequences of vector-valued functions which are bounded respectively in $L^p(\Om)^N$ and $L^q(\Om)^N$, with ${1/p}+{1/q}=1+{1/(N-1)}$, and whose respectively divergence and curl are compact in suitable spaces. We also assume that the product converges weakly in $W^{-1,1}(\Om)$. The key ingredient of the proof is a compactness result for bounded sequences in $W^{1,q}(\Om)$, based on the imbedding of $W^{1,q}(S_{N-1})$ into $L^{p'}(S_{N-1})$ ($S_{N-1}$ the unit sphere of $\RR^N$) through a suitable selection of annuli on which the gradients are not too high, in the spirit of \cite{DeG1,Man}. The div-curl result is applied to the homogenization of equi-coercive systems whose coefficients are equi-bounded in $L^\rho(\Om)$ for some $\rho>{N-1\over 2}$ if $N>2$, or in $L^1(\Om)$ if $N=2$. It also allows us to prove a weak continuity result for the Jacobian for bounded sequences in $W^{1,N-1}(\Om)$ satisfying an alternative assumption to the $L^\infty$-strong estimate of~\cite{BrNg}.
Two examples show the sharpness of the results.
\end{abstract}
\par\bs\noindent
\tableofcontents
\section{Introduction}
In the early 1970s Murat and Tartar noticed that for any sequence $\sigma_n$ weakly converging to $\sigma$ in $L^p_{\rm loc}(\RR^N)$, $N\geq 2$ and $p\in(1,\infty)$, and any sequence $u_n$ converging weakly to $u$ in $W^{1,p'}_{\rm loc}(\RR^N)$ such that $\div\,\sigma_n$ converges strongly in $W^{-1,p}_{\rm loc}(\RR^N)$, a simple integration by parts leads to the convergence
\beq\label{dcel}
\sigma_n\cdot\nabla u_n\rightharpoonup\sigma\cdot\nabla u\quad\mbox{in }\D'(\RR^N).
\eeq
They extended this remark to the more general case where $\nabla u_n$ is replaced by any sequence $\eta_n$ such that $\curl\,\eta_n$ is compact in $W^{-1,p'}_{\rm loc}(\RR^N)$ (see \cite{Mur2}). The successful compensated compactness theory was born with a fruitful application to homogenization theory \cite{Mur}.
\par
Actually, the elementary div-curl \eqref{dcel} contains hidden informations. Indeed, Coifman {\em et al.} proved that if $\div\,\sigma$ is in $W^{-1,s}_{\rm loc}(\RR^N)$ with $s>p$, then $\sigma\cdot\nabla u$ belongs to the Hardy space $\H^1_{\rm loc}(\RR^N)$.
More recently, Conti {\em et al}. \cite{CDM} obtained a new div-curl result relaxing the compensation conditions on $\div\,\sigma_n$ and $\curl\,\eta_n$ to the space $W^{-1,1}_{\rm loc}(\RR^N)$, but assuming that the sequence $\sigma_n\cdot\eta_n$ is equi-integrable.
\par
On the other hand, in the spirit of \cite{Mur,Mur2} and using an appropriate Hodge decomposition of vector-valued fields, it was proved in \cite{BCM} that, given a bounded open set $\Om$ of $\RR^N$, $N\geq 2$, if $p,q\in [1,\infty)$ satisfy
\beq\label{pqNi}
p,q\geq 1,\quad {1\over p}+{1\over q}\leq 1+{1\over N},
\eeq
and if $\sigma_n$, $\eta_n$ satisfy the convergences
\beq\label{cvsinetani}
\sigma_n\rightharpoonup\sigma\;\;
\left\{\ba{ll}
L^p(\Om)^N, & \mbox{if }p>1
\\ \ecart
\M(\Om)^N\,*, & \mbox{if }p=1,
\ea\right.
\quad
\eta_n\rightharpoonup\eta\;\;
\left\{\ba{ll}
L^q(\Om)^N, & \mbox{if }q>1
\\ \ecart
\M(\Om)^N\,*, & \mbox{if }q=1,
\ea\right.
\eeq
and the compensation conditions
\beq\label{dsincetani}
\div\,\sigma_n\to\div\;\sigma\ 
\left\{\ba{ll}
W^{-1,q'}(\Om)^N, & \mbox{if }q>1
\\ \ecart
L^N(\Om)^N, & \mbox{if }q=1,
\ea\right.
\quad
\curl\,\eta_n\to\curl\,\eta\ 
\left\{\ba{ll}
W^{-1,p'}(\Om)^N, & \mbox{if }p>1
\\ \ecart
L^N(\Om)^N, & \mbox{if }p=1,
\ea\right.
\eeq
then there exist two sequences $x_j$ in $\Om$ and $c_j$ in $\RR^N$ such that
\beq\label{dcBCM}
\sigma_n\cdot\eta_n\rightharpoonup\sigma\cdot\eta+\sum_{j=1}^\infty\div\,(c_j\,\delta_{x_j})\quad\mbox{in }\D'(\Om).
\eeq
\par
In this paper we generalize the div-curl result of \cite{Mur2,Tar2,BCM} assuming instead of \eqref{pqNi} the inequality
\beq\label{pqN-1i}
{1\over p}+{1\over q}\leq 1+{1\over N-1}.
\eeq
The statement type is given by the following result which is refined in Theorem~\ref{thmdc} (strict inequality in \eqref{pqN-1i}) below:
\begin{Thm}\label{thmdci}
Assume that \eqref{pqN-1i} holds with the strict inequality. Consider two sequences $\sigma_n$ in $L^p(\Om)^N$ and $\eta_n$ in $L^{p'}(\Om)^N$ satisfying convergences \eqref{cvsinetani}, \eqref{dsincetani} with $\sigma\in L^p(\Om)^N$ and $\eta\in L^{p'}(\Om)^N$, and such that
\beq\label{dcBC}
\sigma_n\cdot\eta_n\mbox{ converges weakly  in }W^{-1,1}(\Om)^N.
\eeq
Then, the weak limit of $\sigma_n\cdot\eta_n$ is $\sigma\cdot\eta$.
\end{Thm}
\noindent
When equality holds in \eqref{pqN-1i}, Theorem~\ref{thmdci} is also extended to Theorem~\ref{calipn1} (case $p>1$) and to Theorem~\ref{calipi1} (case $p=1$), under some equi-integrability assumption on $|\eta_n|$. Moreover, a counterexample to the div-curl result is given when this equi-integrability condition does not hold (see Proposition~\ref{procex} below).
\par\bs
The proof of Theorem~\ref{thmdci} differs notably from the ones of \cite{Mur2,Tar2,BCM}. In fact, the improvement from the bound ${1/N}$ to the bound ${1/(N-1)}$ is connected to the imbedding, related to the unit sphere $S_{N-1}$ of $\RR^N$, of $W^{1,q}(S_{N-1})$ into $L^{p'}(S_{N-1})$, which is compact when inequality \eqref{pqN-1i} is strict. Our approach is inspired by both
\begin{itemize}
\item[-] De Giorgi's method \cite{DeG1} for matching boundary values in $\Ga$-convergence, which consists in finding suitable annuli where the energy does not concentrate,
\item[-] Manfredi's method \cite{Man} for proving the continuity of a weakly monotone ({\em i.e.} satisfying a maximum principle) function in $W^{1,m}$, with $m>N-1$, which consists in selecting spheres on which the gradient of the function is not too high.\footnote{Manfredi's method was used in \cite{BrCa3} to derive, thanks to the maximum principle, a uniform convergence result for sequences of solutions to elliptic equations with non equi-bounded coefficients. But of course this approach cannot be extended to elliptic systems.}
\end{itemize}
Then, the key ingredient of the proof of Theorem~\ref{thmdci} is given by the following result refined in Lemma~\ref{lemLpW1q} below:
\begin{Lem}\label{lemi}
Let $N\geq 2$, $0<R_0<R$, and $q>1$. Consider a sequence $u_n$ which converges weakly to $u$ in $W^{1,q}\big(\{R_0\!<\!|x|<\!R\}\big)$.
Then, there exists a closed set $U_n\subset(R_0,R)$, whose measure is arbitrarily close to $R-R_0$, such that
\[
\left\{\ba{lll}
\dis \sup_{r\in U_n}\left(\int_{S_{N-1}}\big|u_n(ry)-u(ry)\big|^s\,ds(y)\right)\to 0,
& 1\leq s<q^*_{N-1}=\tex{\big({1\over q}-{1\over N-1}\big)^{-1}}, & \hbox{if }q<N-1
\\ \ecart
\dis \sup_{r\in U_n}\left(\int_{S_{N-1}}\big|u_n(ry)-u(ry)\big|^s\,ds(y)\right)\to 0,
& 1\leq s<\infty & \hbox{if }q=N-1
\\ \ecart
\dis \sup_{r\in U_n}\left(\sup_{y\in S_{N-1}}\big|u_n(ry)-u(ry)\big|\right)\to 0, & & \hbox{if }q>N-1.
\ea\right.
\]
\end{Lem}
\noindent
Lemma~\ref{lemi} means that one can select a $n$-dependent set $U_n$ of annuli on which a strong estimate of $u_n-u$ holds. This set is built from not too high values of $|\nabla u_n|$ (see the definition \eqref{l2pr2Un} of $U_n$).
Lemma~\ref{lemi} also extends to Lemma~\ref{le1cc} in connection with Theorem~\ref{calipn1} (critical case $s=q^*_{N-1}$), and to Lemma~\ref{le2cc} (case $q=N-1$, with a uniform convergence result) in connection with Theorem~\ref{calipi1}, under a suitable equi-integrability assumption on $|\nabla u_n|$.
\par\bs
Beyond H-convergence for sequences of conductivity equations \cite{Mur}, which is historically linked to the classical div-curl lemma of \cite{Mur2}, Tartar \cite{Tar2} extended its application field to various pde's including the hyperbolic equations. In the spirit of H-convergence, the div-curl approach was applied to linear elasticity in \cite{Fra}. The seminal works \cite{Spa,Mur} on homogenization of elliptic problems are based on the boundedness (from below and above) of the sequences of coefficients involving in the equations. More recently, the boundedness assumption has been relaxed thanks to an appropriate extension of the div-curl lemma in conductivity \cite{Bri2,BrCa,BCM}, and in elasticity~\cite{BrCE1}. In these works the dimension is $N=2$, and the sequences of coefficients are assumed to be uniformly bounded in $L^1$. The $L^1$-boundedness condition has been removed in the setting of the homogenization of linear and nonlinear scalar problems \cite{BrCa1,BBC,BrCa3} using the maximum principle in an essential way. Up to our knowledge, except the recent approach of \cite{BrCa4} which is however based on a quite restrictive equi-integrability condition, the only available tool for deriving compactness results in the homogenization of sequences of systems with $L^1$-bounded coefficients remains the div-curl lemma. So, the linear elasticity result \cite{BrCE1} shows that in dimension two the violation of the $L^1$-bound may induce second gradient terms in the homogenized equation. This anomalous behavior was previously observed in \cite{PiSe} in dimension three with a two-scale approach. In fact, the situation in three-dimensional linear elasticity is much more intricate since the closure set of equations is very large as shown in \cite{CESe2}, while it is limited by the Beurling-Deny representation formula \cite{BeDe} in the conductivity case \cite{CESe1}. In view of the compactness result of \cite{CaSb} {\em versus} the nonlocal effects obtained in \cite{FeKh,Khr,BeBo,BrTc,CESe1} and naturally connected with the Beurling-Deny formula by \cite{Mos}, the good assumption to avoid any loss of compactness in the homogenization process seems to be, at least in the scalar case and in any dimension, the equi-boundedness and the equi-integrability in $L^1$ of the sequences of coefficients.
\par
In this context and as a by-product of the div-curl result of Theorem~\ref{thmdci} and its extensions, we have the following homogenization result which is refined in Theorem~\ref{thmGc} (with a $\Gamma$-convergence approach), and in Theorem~\ref{thmHc} (with a H-convergence approach) below:
\begin{Thm}\label{thm.Gci}
Let $\Om$ be a bounded open set of $\RR^N$, $N\geq 2$, and let $M$ be a positive integer.
Consider a non-negative symmetric tensor-valued function $\AA_n$ in $L^\infty(\Om)^{(M\times N)^2}$ such that there exists a constant $\al>0$ satisfying
\beq\label{ellAali}
\al\int_\Om\,|Dv|^2\,dx\leq\int_\Om\AA_n Dv:Dv\,dx,\quad\forall\,v\in H^1_0(\Om)^M,
\eeq
and such that
\beq\label{cvAnLai}
|\AA_n|\mbox{ is bounded in}\quad
\left\{\ba{lll}
L^1(\Om), & & \mbox{if }N=2
\\ \ecart
L^\rho(\Om), & \dis \mbox{ with }\rho>{N-1\over 2}, & \mbox{if }N>2.
\ea\right.
\eeq
Then, there exist a subsequence of $n$, still denoted by $n$, and a non-negative symmetric tensor-valued $\AA\in\M(\Om)^{(M\times N)^2}$ if $N=2$, or $\AA\in L^{\rho}(\Om)^{(M\times N)^2}$ if $N>2$, satisfying \eqref{ellAali}, such that the following $\Ga$-convergence for the $L^2(\Om)^M$ strong topology holds
\[
\left\{\ba{ll}
\dis \left(v\in H^1_0(\Om)^M\mapsto\int_\Om\AA_n Dv:Dv\,dx\right)
\stackrel{\Ga}{\to}
\left(v\in C^1_0(\Om)^M\mapsto\int_\Om\AA Dv:Dv\,dx\right), & \mbox{if }N=2
\\ \ecart
\dis \left(v\in H^1_0(\Om)^M\mapsto\int_\Om\AA_n Dv:Dv\,dx\right)
\stackrel{\Ga}{\to}
\left(v\in W^{1,{2\rho\over\rho-1}}_0(\Om)^M\mapsto\int_\Om\AA Dv:Dv\,dx\right), & \mbox{if }N>2.
\ea\right.
\]
\end{Thm}
Note that in dimension three the result of Theorem~\ref{thm.Gci} holds if the sequence $|\AA_n|$ is bounded in some $L^\rho$ space with $\rho>1$. This condition is stronger than the equi-integrability of the coefficients, but is not so far from it. Alternatively, assuming that $\AA_n$ is close in $L^1$-norm to an equi-coercive and equi-bounded sequence $\BB^n$,  we have obtained in \cite{BrCa4} a similar compactness result by a quite different approach. Also note that the two-dimensional case of Theorem~\ref{thm.Gci} includes the homogenization results of \cite{BrCa,BrCa1,BrCE1}.
\par\bs
The classical div-curl lemma and more generally the compensated compactness has been successively used for weak continuity problems \cite{Mur,Mur2,Tar}, and in particular for the weak continuity of the Jacobian in connection with the calculus of variations \cite{Mor,Res,Ball,BaMu,Dac}. The divergence formulation of the Jacobian, denoted as Det, was originally established by Morrey~\cite{Mor}, and leads to the classical weak continuity result (see, {\em e.g.}, \cite{Ball,DaMu,IwSb,FLM}): for any regular open bounded set $\Om$ of $\RR^N$, $N\geq 2$, and for any $s>{N^2\over N+1}$,
\beq\label{cvDeti}
u_n\rightharpoonup u\;\;\mbox{in }W^{1,s}(\Om)^N\;\;\Rightarrow\;\;{\rm Det}\,(Du_n)\rightharpoonup{\rm Det}\,(Du)\;\;\mbox{in }\D'(\Om).
\eeq
Up to our knowledge, the most recent improvement of \eqref{cvDeti} is due to Brezis and Nguyen \cite{BrNg} who have obtained the weak continuity result
\beq\label{BNdeti}
\left.\ba{lll}
u_n\rightharpoonup u & \mbox{in }W^{1,N-1}(\Om)^N
\\ \ecart
u_n\to u & \mbox{in }L^{\infty}(\Om)^N, & \hbox{if }N=2
\\ \ecart
u_n\to u & \mbox{in }B\!M\!O(\Om)^N, & \hbox{if }N\geq 3
\ea\right\}
\quad\Rightarrow\quad{\rm Det}\,(Du_n)\rightharpoonup{\rm Det}\,(Du)\;\;\mbox{in }\D'(\Om),
\eeq
where the refinement in $B\!M\!O$ is partly based on the div-curl approach of \cite{CLMS}.
Actually, Brezis and Nguyen prove a delicate estimate (see \cite{BrNg}, Theorem~1)   which implies convergence \eqref{BNdeti}.
\par
Using the div-curl result of  Theorem~\ref{calipi1} we prove the alternative weak continuity convergence of the Jacobian in $W^{1,N-1}$ under different assumptions (see Theorem~\ref{thmdet} below for a refined statement):
\begin{Thm}\label{thmdeti}
Let $\Om$ be a bounded open set of $\RR^N$, with $N\geq 2$.
Consider a sequence of vector-valued functions $u_n=\left(u_n^1,\dots,u_n^N\right)$ in $W^{1,N}(\Om)^M$ satisfying
\beq\label{codjuni}
u_n\rightharpoonup u=\left(u^1,\dots,u^N\right)\quad
\left\{\ba{ll}
\hbox{in }W^{1,N-1}(\Om)^N, & \mbox{if }N>2
\\ \ecart
\hbox{in }BV(\Om)^N\,*, & \mbox{if }N=2,
\ea\right.
\eeq
\beq\label{thmde1i}
{\rm Det}\,(Du_n)\rightharpoonup \mu\quad\hbox{in }W^{-1,1}(\Om).
\eeq
Also assume that $\nabla u_n^1$ is equi-integrable in the Lorentz space $L^{N-1,1}(\Om)^N$.
\\
Then, the limit distribution $\mu$ is given by the variational formulation
\beq\label{demui}
\langle\mu,\psi\rangle=-\into\left(\sum_{j=1}^k{\rm cof}\,(Du)_{1j}\,\partial_j\psi\, u^1\right)dx,
\eeq
for a suitable dense set of radial fonctions $\psi$ in $W^{1,\infty}_0(\Om)$.
\end{Thm}
Example~\ref{cexDet} below shows that the loss of equi-integrability for $\nabla u_n^1$ may induce a concentration effect in the weak convergence of the Jacobian. This example also illustrates the sharpness of the weak continuity result of \cite{BrNg}.
\subsubsection*{Notations}
\begin{itemize}
\item $M$ is a positive integer, and $N$ is an integer $\geq 2$.

\item $\left(e_1,\dots,e_N\right)$ denotes the canonical basis of $\RR^N$, and $\left(f_1,\dots,f_M\right)$ the one of $\RR^M$.


\item $:$ denotes the scalar product in $\RR^{M\times N}$, {\em i.e.}  $\xi:\eta=\tr\left(\xi^T\eta\right)$ for any $\xi,\eta\in\RR^{M\times N}$.

\item $B_R$ denotes an open ball of $\RR^N$ centered at the origin zero and of radius $R>0$. The ball centered at the point $x_0$ and of radius $R$ is denoted by $B(x_0,R)$.

\item For $0<R_0<R$, $C(R_0,R)$ denotes the open crown $B_R\setminus \bar B_{R_0}$.

\item $S_{N-1}$ denotes the unit sphere of $\RR^N$ for any integer $N\geq 2$.

\item For any $p\in[1,\infty]$, $p':={p\over p-1}\in[1,\infty]$ denotes the conjugate exponent of $p$.

\item For any $q\in[1,N)$, $q^*_N:=\big({1\over q}-{1\over N}\big)^{-1}$ denotes the critical Sobolev exponent in dimension~$N$.

\item $|E|$ denotes Lebesgue's measure of any measurable set $E\subset\RR^N$. When $E$ is a subset of a manifold of $\RR^N$ of dimension $P\leq N$, $|E|$ is also used to denote the corresponding Hausdorff measure of order $P$.

\item $1_E$ denotes the characteristic function of any set $E$.

\item $\nabla u$ denotes the gradient of the scalar distribution $u:\RR^N\to\RR$.

\item $Du$ denotes the Jacobian matrix of the vector-valued distribution $u:\RR^N\to\RR^M$, {\em i.e.}
\[
Du:=\left[{\partial u_i\over\partial x_j}\right]_{1\leq i\leq M,\,1\leq j\leq N}\in\RR^{M\times N}.
\]

\item $\div$ denotes the classical divergence operator acting on the vector-valued distributions.

\item $\Div$ denotes the vector-valued differential operator taking the divergence of each row of a matrix-valued distribution,
\[
\Div\,V:=\left[\sum_{j=1}^N{\partial V_{ij}\over\partial x_j}\right]_{1\leq i\leq M},\quad\mbox{for }V:\RR^N\to\RR^{M\times N}.
\]

\item $\curl$ denotes the classical curl operator acting on the vector-valued distributions.

\item $\Curl$ denotes the vector-valued differential operator taking the curl of each row of a matrix-valued distribution,
\[
\Curl\,V:=\left[{\partial V_{ij}\over\partial x_k}-{\partial V_{ik}\over\partial x_j}\right]_{1\leq i\leq M,\, 1\leq j,k\leq N},\quad\mbox{for }V:\RR^N\to\RR^{M\times N}.
\]

\item $\M(X)$ denotes the set of the bounded Radon measures on a locally compact set $X$.

\item For a bounded open set $\Om$ of $\RR^N$, $W^{1,\infty}_0(\Om)$ denotes the space of the functions in $W^{1,\infty}(\Om)$ which are equal to $0$ on $\partial\Om$.

\item $W^{-1,1}(\Om)$ denotes the set composed of the divergences of functions in $L^1(\Om)^N$. We can check that the dual of $W^{-1,1}(\Om)$ is $W^{1,\infty}_0(\Om)$ (using essentially the fact that the dual of $L^1(\Om)$ is $L^\infty(\Om)$, and any vector-valued distribution which vanishes on the divergence free functions is a gradient). Hence, the weak convergence of $\mu_n$ to $\mu$ in $W^{-1,1}(\Om)$ reads as
\beq\label{wcW-11}
\langle \mu_n,\ph\rangle\limi\langle \mu,\ph\rangle,\quad\forall\,\ph\in W^{1,\infty}_0(\Om).
\eeq
Note that the weak-$\ast$ convergence in  $\M(\Om)$ implies the weak convergence in $W^{-1,1}(\Om)$.
\end{itemize}
\section{The div-curl result}\label{s.dc}
\subsection{The case: ${1\over p}+{1\over q}<1+{1\over N-1}$}
We have the following div-curl result:
\begin{Thm}\label{thmdc}
Let $\Om$ be a bounded open set of $\RR^N$, with $N\geq 2$, and let $p,q\geq 1$ such that
\beq\label{inepqN-1}
{1\over p}+{1\over q}<1+{1\over N-1}.
\eeq
Consider two sequences of matrix-valued functions $\sigma_n$ and $\eta_n$ such that
\beq\label{siLsetaLs'}
\exists\,s_n\in[p,q'],\quad\sigma_n\in L^{s_n}(\Om)^{M\times N}\;\;\mbox{and}\;\;\eta_n\in L^{s_n'}(\Om)^{M\times N},
\eeq
\beq\label{cvsinetanmu}
\sigma_n:\eta_n\rightharpoonup\mu\quad\mbox{weakly in }W^{-1,1}(\Om).
\eeq
Then, we have the following results according to the cases $p,q>1$, $q=1$ or $p=1$:
\begin{itemize}
\item Assume that $p,q>1$, and that
\beq\label{cvvnpwnq}
\left\{\ba{ll}
\sigma_n\rightharpoonup \sigma & \mbox{in }L^p(\Om)^{M\times N}
\\ \ecart
\eta_n\rightharpoonup \eta & \mbox{in }L^q(\Om)^{M\times N},
\ea\right.
\eeq
\beq\label{dvnqcwnp}
\left\{\ba{ll}
\Div\,\sigma_n\to \Div\,\sigma & \mbox{ in }W^{-1,q'}(\Om)^{M}
\\ \ecart
\Curl\,\eta_n\to \Curl\,\eta & \mbox{ in }W^{-1,p'}(\Om)^{M\times N\times N}.
\ea\right.
\eeq
If the limits $\sigma$ and $\eta$ satisfy condition \eqref{siLsetaLs'}, then
\beq\label{mu=sieta}
\mu=\sigma:\eta.
\eeq
Otherwise, for any function $u$ satisfying
\beq\label{eta-DuLp'}
u\in W^{1,q}(\Om)^M\quad\mbox{and}\quad\eta-Du\in L^{p'}_{\rm loc}(\Om)^{M\times N},
\eeq
the limit $\mu$ satisfies the weak formulation
\beq\label{concte1}
\left\{\ba{l}
\dis \forall\,B(x_0,R)\Subset\Om,\ \forall\,\varphi\in W^{1,\infty} (0,\infty),\mbox{ with }{\rm supp}\,\varphi\subset [0,R],
\\ \ecart
\dis \langle\mu,\psi\rangle=-\,\big\langle{\rm Div}\,\sigma, u\psi\big\rangle
+\int_{B(x_0,R)}\sigma:\big[\eta\psi-D(u\psi)\big]\,dx,
\\ \ecart
\dis\hbox{where }\psi(x):=\varphi(|x-x_0|).
\ea\right.
\eeq

\item Assume that $q=1$, and that
\beq\label{cvvnpwn1}
\left\{\ba{ll}
\sigma_n\rightharpoonup \sigma & \mbox{ in }L^p(\Om)^{M\times N}
\\ \ecart
\eta_n\stackrel{\ast}\rightharpoonup \eta &  \mbox{ in }\M(\Om)^{M\times N},
\ea\right.
\eeq
\beq\label{dvn1cwnp}
\left\{\ba{ll}
\Div\,\sigma_n\to \Div\,\sigma & \mbox{ in }L^N(\Om)^M
\\ \ecart
{\rm Curl}\, \eta_n\to {\rm Curl}\,\eta & \mbox{ in }W^{-1,p'}(\Om)^{M\times N\times N}.
\ea\right.
\eeq
If the limits $\sigma$ and $\eta$ satisfy condition \eqref{siLsetaLs'}, then equality \eqref{mu=sieta} holds.
\par
Otherwise, for any function $u$ satisfying \eqref{eta-DuLp'}, the limit $\mu$ still satisfies the weak formulation \eqref{concte1}.

\item Assume that $p=1$, and that
\beq\label{cvvnpwn13}
\left\{\ba{ll}
\sigma_n\stackrel{\ast}\rightharpoonup \sigma & \mbox{ in }\M(\Om)^{M\times N}
\\ \ecart
\eta_n\rightharpoonup \eta &  \mbox{ in }L^q(\Om)^{M\times N},
\ea\right.
\eeq
\beq\label{dvn1cwnp3}
\left\{\ba{ll}
{\rm Div}\,\sigma_n\to {\rm Div}\,\sigma & \mbox{ in }W^{-1,q'}(\Om)^{M}
\\ \ecart
{\rm Curl}\, \eta_n\to {\rm Curl}\,\eta & \mbox{ in }L^N(\Om)^{M\times N\times N}.
\ea\right.
\eeq
If the limits $\sigma$ and $\eta$ satisfy condition \eqref{siLsetaLs'}, then equality \eqref{mu=sieta} holds.
\par
Otherwise, for any function $u$ satisfying
\beq\label{eta-DuW1N}
u\in W^{1,q}(\Om)^M\quad\mbox{and}\quad\eta-Du\in W^{1,N}_{\rm loc}(\Om)^{M\times N},
\eeq
the limit $\mu$ satisfies the weak formulation
\beq\label{concte3}
\left\{\ba{l}\dis\forall\,B(x_0,R)\Subset\Om,\ \forall\,\varphi\in W^{1,\infty}(0,\infty),\mbox{ with }{\rm supp}\,\varphi\subset [0,R],\mbox{ such that }
\\ \ecart\dis
\exists\,U\mbox{closed set of }[0,R],\mbox{ with }
\left|\ba{l}
(r,y)\mapsto u(x_0+ry)\in C^0\big(U;W^{1,q}(S_{N-1})\big)
\\
\ph'\mbox{ is continuous on $U$ with support in }U,
\ea\right.
\\ \ecart
\dis \langle\mu,\psi\rangle=-\,\big\langle{\rm Div}\,\sigma, u\psi\big\rangle
+\int_{B(x_0,R)}\sigma(dx):\big[\eta\psi-D(u\psi)\big],
\\ \ecart
\dis\hbox{where }\psi(x):=\varphi(|x-x_0|).
\ea\right.
\eeq
\end{itemize}
\end{Thm}
First of all, focus on the case $p,q>1$:
\begin{Rem} \label{rem1tp}
First of all, in view of the weak formulation~\eqref{concte1} note that
\[
\sigma:\big[\eta\psi-D(u\psi)\big]=\sigma:(\eta-Du)\,\psi-(\sigma\nabla\psi)\cdot u.
\]
Hence, since $\sigma:(\eta-Du)$ is in $L^1(\Om)$ by \eqref{eta-DuLp'}, the last integral term of \eqref{concte1} has a sense if and only if the integral term
\[
\int_{B(x_0,R)}(\sigma\nabla\psi)\cdot u\,dx,
\]
has a sense. This needs radial test functions $\psi$ and will be discussed in the general setting of Remark~\ref{rem.nregsinetan} below.
However, observe that the set of functions $\psi$ of the form
\[
\psi(x)=\sum_{i=1}^m c_i\,\varphi_i(|x-x_i|)
\]
such that for any $m\geq 1$ and $i\in\{1,\dots,m\}$, $c_i$ is a real constant, $x_i\in\Om$ and $\varphi_i\in W^{1,\infty}(0,\infty)$ with ${\rm supp}\,(\varphi_i)\subset [0,R_i]$, where $R_i>0$ and $B(x_i,R_i)\Subset\Om$, is dense in $W^{1,\infty}_0(\Om)$.
Therefore, the weak formulation \eqref{concte1} fully characterizes the distribution $\mu$.
\par
On the other hand, the existence of a function $u$ satisfying \eqref{eta-DuLp'} follows from the fact that $\Curl\,\eta$ belongs to $W^{-1,p'}(\Om)^{M\times N\times N}$ (see, {\em e.g.}, \cite{BCM}, Proposition~2.5). Note that \eqref{concte1} does not depend actually on the choice of the function $u$ satisfying \eqref{eta-DuLp'}. Indeed, let $u$ and $\tilde u$ be two functions satisfying \eqref{eta-DuLp'}. Since $u-\tilde u\in W^{1,q}(\Om)^N\cap W^{1,p'}_{\rm loc}(\Om)^N$, we have
\[
-\,{\rm Div}\,\sigma\cdot(u-\tilde u)-\si:D(u-\tilde u)+\div\left(\si^T(u-\tilde u)\right)=0\quad\mbox{in }\Om,
\]
which implies that the right-hand side of \eqref{concte1} is equal to zero with $u-\tilde u$ instead of $u$.
\end{Rem}

\begin{Rem} \label{rem4tp} It is clear that Theorem~\ref{thmdc} implies the classical div-curl result
of \cite{Mur2}, \cite{Tar2}, {\em i.e.} assuming that for $p\in (1,\infty)$,
\[
\left\{\ba{ll}
\sigma_n\rightharpoonup \sigma & \mbox{ in }L^p(\Om)^{M\times N}
\\ \ecart
\eta_n\rightharpoonup \eta & \mbox{ in }L^{p'}(\Om)^{M\times N},
\ea\right.
\qquad
\left\{\ba{ll}
\Div\,\sigma_n\to \Div\,\sigma & \mbox{ in }W^{-1,p}(\Om)^{M}
\\ \ecart
\Curl\,\eta_n\to \Curl\,\eta & \mbox{ in }W^{-1,p'}(\Om)^{M\times N\times N}.
\ea\right.
\]
then the following convergence holds true
\[
\sigma_n:\eta_n\stackrel{\ast}\rightharpoonup \sigma:\eta\quad\hbox{ in }\M(\Om).
\]
\par
We can also compare our result with the div-curl result of \cite{BCM} based on the convergences \eqref{cvvnpwnq} and \eqref{dvnqcwnp} together with condition
\beq\label{inepqN}
{1\over p}+{1\over q}\leq 1+{1\over N}.
\eeq
First, by Proposition~2.5 of \cite{BCM} there exists a matrix-valued function $\zeta$ such that
\beq\label{si-zeLpLq'}
\zeta\in L^p(\Om)^{M\times N},\quad {\rm Div}\,\zeta=0\;\;\mbox{in }\Om,\quad\sigma-\zeta\in L^{q'}_{\rm loc}(\Om)^{M\times N}.
\eeq
Then, in the case $p,q>1$ (but the other cases are similar), inequality \eqref{inepqN} combined with the Sobolev imbedding $W^{1,q}(\Om)\hookrightarrow L^{q^*_N}(\Om)$ implies that if the functions $u$ and $\zeta$ satisfy \eqref{eta-DuLp'} and \eqref{si-zeLpLq'}, then $\zeta^Tu$ is in $L^1_{\rm loc}(\Om)^N$. Therefore, using that $\zeta$ is divergence free, the limit formulation \eqref{concte1} can be written
\beq\label{wkmu}
\mu=\sigma:(\eta-Du)+(\sigma-\zeta):Du+{\rm div}\,(\zeta^Tu)\quad\mbox{in }\D'(\Om),
\eeq
which is the weak formulation for $\sigma:\eta$ according to Proposition 2.5 of \cite{BCM}.
However, Theorem~2.3 of \cite{BCM} shows for sequences $\sigma_n$ and $\eta_n$ satisfying \eqref{siLsetaLs'}, \eqref{cvvnpwnq}, \eqref{dvnqcwnp}, the existence of two sequences $x_j$ in $\Om$ and $c_j$ in $\RR^N$ such that
\[
\sigma_n:\eta_n\rightharpoonup\mu+\sum_{j=1}^\infty{\rm div}\,(c_j\,\delta_{x_j})\quad\hbox{ in }\D'(\Om).
\]
The reason of this apparent contradiction with equality \eqref{mu=sieta} is that in Theorem~\ref{thmdc} we have also assumed that $\sigma_n:\eta_n$ converges weakly in $W^{-1,1}(\Om)$, while in \cite{BCM} the convergence of $\sigma_n:\eta_n$ is obtained in the (larger) distributions space. It is easy to see that $\sigma_n:\eta_n$ in \cite{BCM} is actually the divergence of a sequence which converges only in the weak-$*$ sense of the measures.
\end{Rem}
\begin{Rem}\label{rem.nregsinetan}
In view of \eqref{eta-DuLp'} and \eqref{concte1} the regularity assumption \eqref{siLsetaLs'} for $\sigma_n$ and $\eta_n$, which holds in most situations, can be replaced in the case $p,q>1$ by the more general conditions:
\beq\label{sinetanW1}
\sigma_n:\eta_n\in W^{-1,1}(\Om),
\eeq
and similarly to \eqref{concte1}, for any $u_n\in W^{1,q}_{\rm loc}(\Om)^M$ satisfying $\eta_n-Du_n\in L^{p'}_{\rm loc}(\Om)^{M\times N}$, we have
\beq\label{concte2}
\left\{\ba{l}
\dis \forall\,B(x_0,R)\Subset\Om,\ \forall\,\varphi\in W^{1,\infty} (0,\infty),\mbox{ with }{\rm supp}\,\varphi\subset [0,R],
\\ \ecart
\ba{ll}
\dis \big\langle\sigma_n:\eta_n,\psi\big\rangle= &
\dis -\,\big\langle{\rm Div}\,\sigma_n, u_n\psi\big\rangle+\int_{B(x_0,R)}\sigma_n:\big[\eta_n\psi-D(u_n\psi)\big]\,dx
\\ \ecart
& \dis -\,\big\langle{\rm Div}\,\sigma_n, u_n\psi\big\rangle+\int_{B(x_0,R)}\big[\sigma_n:(\eta_n-Du_n)\,\psi-(\sigma_n\nabla\psi)\cdot u_n\big]\,dx,
\ea
\\ \ecart
\dis \hbox{where }\psi(x):=\varphi(|x-x_0|).
\ea\right.
\eeq
So the distribution $\sigma_n:\eta_n$ is defined by the formula \eqref{concte2}, and its extension to $W^{-1,1}(\Om)$ is required through condition \eqref{sinetanW1}.
\par
Then, we need to justify the integral term of \eqref{concte2}
\[
\int_{B(x_0,R)}(\sigma_n\nabla\psi)\cdot u_n\,dx.
\]
To this end, note that $u_n\in W^{1,q}_{\rm loc}(\Om)^M$ implies that
\[
\ba{lcll}
v_n: & (0,R)\times S_{N-1} & \to & \RR^M
\\[1.mm]
& (r,y) & \mapsto & u_n(x_0+ry)
\ea
\]
is in $L^q_{r^{N-1}dr}(0,R;W^{1,q}(S_{N-1}))^M$, and thus by Sobolev's imbedding, in $L^q_{r^{N-1}dr}(0,R;L^{p'}(S_{N-1}))^M$ due to \eqref{inepqN-1}. Hence, at least for $\varphi\in W^{1,\infty} (0,\infty)$, with ${\rm supp}\,\varphi\subset [0,R]$ and
\beq\label{condsov}
{\rm supp}\,(\varphi')\subset U_\la:=\left\{r\in (0,R):\int_{\partial B(x_0,r)}|u_n|^{p'}\,ds(x)\leq\lambda\right\}\quad\hbox{ for some }\lambda>0,
\eeq
we deduce that the right-hand side of \eqref{concte2} has a sense. But if \eqref{concte2} holds at least for $\varphi$ satisfying 
\eqref{condsov}, then using that $\sigma_n:\eta_n$ is in $W^{-1,1}(\Om)$ the function
\[
g_n:r\mapsto r^{N-1}\int_{S_{N-1}}\big(\si_n(x_0+ry)\,y\big)\cdot u_n(x_0+ry)\,ds(y)
\]
satisfies, by \eqref{concte2} together with the definition of $W^{-1,1}$, the equality
\[
\int_0^R\varphi'\,g_n\,dr=\int_0^R\varphi\,f_n\,dr+\int_0^R\varphi'\,h_n\,dr,\quad\mbox{where }f_n,h_n\in L^1(0,R),
\]
which implies that for any $\varphi\in W^{1,\infty} (0,\infty)$, with ${\rm supp}\,\varphi\subset [0,R]$,
\[
\int_0^R\varphi'\,1_{U_\la}\,g_n\,dr=\int_0^R\left(\int_0^r\varphi'\,1_{U_\la}\right)f_n\,dr
+\int_0^R\varphi'\,1_{U_\la}\,h_n\,dr.
\]
This combined with $|U_{\la}|\to R$ as $\la\to\infty$, allows us to conclude that $g_n$ is in $L^1(0,R)$.
Therefore, the weak formulation \eqref{concte2} is actually satisfied for any $\varphi\in W^{1,\infty} (0,\infty)$, with ${\rm supp}\,\varphi\subset [0,R]$.
\par
The same argument applies to the limit formulation \eqref{concte1}. Moreover, following the first argument of Remark~\ref{rem1tp} the weak formulation \eqref{concte1} fully characterizes the distribution $\mu$.
\end{Rem}
The case $q=1$ is similar to the case $p,q>1$. Now, focus on the case $p=1$ which is more delicate concerning the sense of the weak formulation \eqref{concte3}:
\begin{Rem}\label{rem.p=1}
Assume that $p=1$, and thus by \eqref{inepqN-1} $q>N-1$.
With respect to the first term in the right-hand side of \eqref{concte3}, since ${\rm Div}\,\sigma$ is in $W^{-1,q'}(\Om)^{M}$, there exists a matrix-valued Radon measure $\zeta$ satisfying (see \cite{BCM}, Proposition~2.5)
\beq\label{si-zeMLq'}
\zeta\in\M(\Om)^{M\times N},\quad {\rm Div}\,\zeta=0\;\;\mbox{in }\Om,\quad\sigma-\zeta\in L^{q'}_{\rm loc}(\Om)^{M\times N}.
\eeq
Thanks to a result due to Bourgain, Brezis \cite{BoBr}, the two first assertions of \eqref{si-zeMLq'} imply that the measure $\zeta$ is actually in $W^{-1,N'}_{\rm loc}(\Om)^{M\times N}$. Hence, it follows from \eqref{eta-DuW1N} that
\beq\label{si-zeDu}
\sigma:(\eta-Du)=(\sigma-\zeta):(\eta-Du)+\zeta:(\eta-Du)\in L^1_{\rm loc}(\Om)+W^{-1,1}_{\rm loc}(\Om),
\eeq
which yields a sense to the integral term
\[
\int_{B(x_0,R)}\sigma(dx):(\eta-Du)\,\psi.
\]
\par
With respect to the last term in the right-hand side of  \eqref{concte3}, observe that the function $v:(0,R)\times S_{N-1}\to \RR^M$ defined by $v(r,y):=u(x_0+ry)$ belongs to $L^q_{r^{N-1}dr}(0,R;W^{1,q}(S_{N-1}))^M$, and thus to $L^q_{r^{N-1}dr}(0,R;C^0(S_{N-1}))^M$ by Sobolev's imbedding due to $q>N-1$.
Then, by Lusin's theorem, for any $\ep>0$, there exists of a closed set $U$ satisfying the second line of \eqref{concte3} such that $|U|>R-\ep$.
For such a set $U$, the function $v$ is in $C^0(U\times S_{N-1})$ (again by Sobolev's imbedding) and $u$ is thus continuous on the closed set
\[
K:=\big\{x\in \bar\Om:x=x_0+ry,\ r\in U, \ y\in S_{N-1}\big\},
\]
so that $\nabla\psi\otimes u$ can be extended to a continuous function in $\bar{\Om}$.
Therefore, the last term of~\eqref{concte3}, or equivalently,
\[
\int_{B(x_0,R)}\big[\sigma:(\eta-Du)\,\psi-(\sigma\nabla\psi)\cdot u\big]\,dx,
\]
in which
\[
\int_{B(x_0,R)}(\sigma\nabla\psi)\cdot u\,dx=\int_K(\nabla\psi\otimes u):d\sigma,\quad\mbox{where }\psi(x):=\varphi_U(|x-x_0|),
\]
has a sense for any function $\varphi_U$ satisfying the two first lines of \eqref{concte3}.
Moreover, since $|U|$ can be chosen arbitrarily close to $R$, any $\varphi\in W^{1,\infty}(0,\infty)$, with ${\rm supp}\,\varphi\subset [0,R]$, can be approximated for the weak-$*$ topology of $W^{1,\infty}(0,\infty)$ by a sequence of functions
\[
\varphi_U(r):=\int_R^r\varphi'\,1_U\,ds,\quad\mbox{for }r\in[0,\infty).
\]
But it is not clear that the sole condition $\varphi\in W^{1,\infty}(0,\infty)$, with ${\rm supp}\,\varphi\subset [0,R]$, is sufficient.
\par
Finally, this combined with the first argument of Remark~\ref{rem1tp} implies that the weak formulation \eqref{concte3} fully characterizes the distribution $\mu$.
\end{Rem}
\subsection{Proof of Theorem~\ref{thmdc}}
The key ingredient of the proof of Theorem~\ref{thmdc} is the following compactness result:
\begin{Lem}\label{lemLpW1q}
Let $N\geq 2$, $0<R_0<R$, and $q\geq 1$. Consider a sequence $u_n$ in $W^{1,q}(C(R_0,R))^M$ such that
\beq\label{hlecnc1}
\left\{\ba{lll}
u_n\rightharpoonup u & \hbox{in }W^{1,q}\big(C(R_0,R)\big)^M, & \hbox{if }q>1
\\ \ecart
\dis u_n\stackrel{\ast}\rightharpoonup u & \hbox{in }BV\big(C(R_0,R)\big)^M, & \hbox{if }q=1.
\ea\right.
\eeq
Define $v_n,v\in L^q(R_0,R;W^{1,q}(S_{N-1}))^M$, or $v\in L^1(R_0,R;BV(S_{N-1}))^M$ if $q=1$, by
\beq\label{devn}
v_n(r,y):=u_n(ry),\quad v(r,y):=u(ry),\quad\hbox{ a.e. }(r,y)\in (R_0,R)\times S_{N-1},
\eeq
and the space $X$ of functions in $S_{N-1}$ by
\beq\label{deX}
X:=\left\{\ba{ll}\dis L^s(S_{N-1})^M,\;\;\hbox{with }1\leq s<q^*_{N-1}=\tex{\big({1\over q}-{1\over N-1}\big)^{-1}}, & \hbox{if }q<N-1
\\ \ecart
\dis L^s(S_{N-1})^M,\;\;\hbox{with }1\leq s<\infty, & \hbox{if }q=N-1
\\ \ecart
\dis C^0(S_{N-1})^M, & \hbox{if }q>N-1.
\ea\right.
\eeq
Moreover, for any $\lambda>0$ and any closed set $U$ of $[R_0,R]$ such that $v\in C^0(U; W^{1,q}(S_{N-1}))^M$ if $q>1$ or $v\in C^0(U; BV(S_{N-1}))^M$ if $q=1$, define the subset $U_n$ of $U$ by
\beq\label{l2pr2Un}
U_n:=\left\{r\in U:\int_{S_{N-1}}\big(|D u_n(ry)|^q+|D u(ry)|^q\big)\,ds(y)\leq\lambda\right\}.
\eeq
Then, we have
\beq\label{l1pr1Un} 
|U\setminus U_n|\leq{1\over \lambda\,R_0^{N-1}}\int_{\{|x|\in U\}}\big(|D u_n|^q+|D u|^q\big)\,dx,
\eeq
\beq\label{l2pr3Un}
\left\{\ba{lll}
\|v_n-v\|_{C^0(U_n;X)}\to 0, & & \hbox{if }q>1
\\ \ecart
\dis \|v_n-v\|_{L^s(U_n;X)}\to 0, & \forall\,s\in[1,\infty), & \hbox{if }q=1.
\ea\right.\eeq
\end{Lem}
\noindent
\begin{proof}
Property \eqref{l1pr1Un} is an immediate consequence of the definition \eqref{l2pr2Un}  of $U_n$. Thus, we just need to prove \eqref{l2pr3Un}.
\par
On the one hand, since $W^{1,q}(S_{N-1})^M$ if $q>1$, or $BV(S_{N-1})^M$ if $q=1$, is compactly imbedded into $X$,
we deduce from Lemma~5.1 of \cite{Lio} that for any $\de>0$, there exists a constant $C_\de>0$ such that 
\[
\left\{\ba{lll}
\left\|w\right\|_{X}\leq C_\de\left\|w\right\|_{L^q(S_{N-1})^M}+\de\,\|D_\tau w\big\|_{L^q(S_{N-1})^{M\times N}},
& \forall\, w\in W^{1,q}(S_{N-1})^M, & \hbox{if }q>1,
\\ \ecart
\left\|w\right\|_{X}\leq C_\de\left\|w\right\|_{L^1(S_{N-1})^M}+\de\,\|D_\tau w\big\|_{\M(S_{N-1})^{M\times N}},
& \forall\, w\in BV(S_{N-1})^M, & \hbox{if }q=1,
\ea\right.
\]
where $D_\tau$ denotes the tangential derivative along the manifold $S_{N-1}$.
Applying these inequalities to $(v_n-v)(r,\cdot)$, and  taking into account the definition \eqref{l2pr2Un} of $U_n$,  we get 
\beq\label{estLthzn}
\left\{\ba{ll} \|v_n-v\|_{C^0(U_n;X)}\leq C_\delta\,\|v_n-v\|_{C^0(U_n;L^q(S_{N-1}))^M}+\delta\lambda^{1\over q}
& \mbox{if }q>1
\\ \ecart
\dis \|v_n-v\|_{L^{s}(U_n;X)}  \leq C_\de\,\|v_n-v\|_{L^{s}(U_n;L^1(S_{N-1}))^M}+\de\la\,|U_n|^{1\over s}
& \mbox{if }q=1.
\ea\right.
\eeq
\par
On the other hand, the sequence $v_n-v$ is bounded in $L^q(R_0,R;W^{1,q}(S_{N-1}))^M$ and the sequence $\partial_r\big(v_n-v\big)$ is bounded in $L^q(R_0,R;L^q(S_{N-1}))^M$ if $q>1$, or in $\M((R_0,R)\times S_{N-1})^M$ if $q=1$. Hence, by a compactness result due to Simon \cite{Sim} (Corollary~8 and Remark~10.1), the sequence $v_n-v$ converges strongly to $0$ in
\beq\label{compS}
\left\{\ba{lll}
C^0\big([R_0,R];L^q(S_{N-1})\big)^M, & & \mbox{if }q>1
\\ \ecart
L^m\big([R_0,R];L^1(S_{N-1})\big)^M, & \forall\,m\in[1,\infty), & \mbox{if }q=1,
\ea\right.
\eeq
which combined with \eqref{estLthzn} yields
\[
\left\{\ba{lll}
\dis \limsup_{n\to\infty}\|v_n-v\|_{C^0(U_n;X)}\leq \delta\lambda^{1\over q}, & \mbox{if }q>1
\\ \ecart
\dis \limsup_{n\to\infty}\|v_n-v\|_{L^{s}(U_n;X)}  \leq \de\la\,R^{1\over s}, & \mbox{if }q=1.
\ea\right.
\]
Finally, the arbitrariness of $\de>0$ leads to \eqref{l1pr1Un}.
\end{proof}
\par
Let us start by the following preliminary remark which illuminates in particular the strategy of the proof of Theorem~\ref{thmdc}.
\begin{Rem}\label{rem3tp}
To fix ideas, assume that $p,q>1$ with \eqref{inepqN-1} (the other cases are similar). As observed in Remark~\ref{rem1tp}, for $\sigma_n\in L^p(\Om)^{M\times N}$ and $\eta_n\in L^q(\Om)^{M\times N}$ such that ${\rm Curl}\,\eta_n$ is in $W^{-1,p'}(\Om)^{M\times N\times N}$, there exists $u_n\in W^{1,q}(\Om)^M$ such that $\eta_n-Du_n\in L^{p'}_{\rm loc}(\Om)^{M\times N}$.
\par
Then, for any $B(x_0,R)\Subset\Om$ and for any $\varphi\in W^{1,\infty}(0,\infty)$ with 
\[
{\rm supp}\,\varphi\subset [0,R],\quad
{\rm supp}\,(\varphi')\subset \left\{r\in [0,R]:\int_{\partial B(x_0,r)}|\nabla u_n|^q\,ds(y)\leq \lambda\right\}\;\;\hbox{for some }\lambda>0,
\]
the integral 
\[
\into\sigma_n:\big[\eta_n\psi-D(u_n\psi)\big]=\into\big[\sigma_n:(\eta_n-Du_n)\,\psi-(\sigma_n\nabla\psi)\cdot u_n\big]\,dx,
\]
where $\psi(x):=\varphi(|x-x_0|)$, is well defined.
Defining $V_n$ as the vector-space spanned by these functions $\psi$, we can then define the linear mapping $F_n:V_n\to\RR$ by 
\beq\label{deffFn}
F_n\psi:=\into\sigma_n:\big[\eta_n\psi-D(u_n\psi)\big].
\eeq
The proof of Theorem~\ref{thmdc} essentially consists in constructing sequences $\psi_n$ in $V_n$ converging to a function $\psi$ in $W^{1,\infty}(\Om)$ weak-$\ast$ such that
\[
F_n\psi_n\to F\psi.
\]
But this does not prove the convergence of $F_n$ to $F$ in any topology because the spaces $V_n$ vary with $n$. This is the reason to make assumption \eqref{cvsinetanmu} in Theorem~\ref{thmdc}.
However, this assumption can be simplified. Indeed, instead of assuming
$\sigma_n:Du_n \in W^{-1,1}(\Om)$, we can assume that
\beq\label{hipfn}
F_n\hbox{ defined by }\eqref{deffFn} \hbox{ can be extended  to an element of }W^{-1,1}(\Om),
\eeq 
 which holds true for example if $\sigma_n^Tu_n$ is in $L^1(\Om)^N$, and then to define $\sigma_n:\eta_n$ in a relaxed way by the equality 
 \beq\label{dersiet}
 \sigma_n:\eta_n:=F_n.
 \eeq
Note that for $\sigma_n$, $\eta_n$ smooth enough this equality holds, but $F_n$ does not necessarily agree with the measurable function $\sigma_n:\eta_n$ which in general is not even in $L^1(\Om)$. Then, also assuming
\beq\label{frdr}
\sigma_n:\eta_n\rightharpoonup \mu\quad\hbox{in }W^{-1,1}(\Om),
\eeq
Theorem~\ref{thmdc} shows that $\mu=\sigma:\eta$, where $\sigma:\eta$ is defined in a relaxed way similarly to $\sigma_n:\eta_n$.
\end{Rem}
The proof of Theorem~\ref{thmdc} will use the following  equi-integrability result for  weakly convergent sequences in $W^{-1,1}(\Om)$ and radial test functions:
\begin{Lem}\label{lem.wcW-11}
Let $x_0\in\Om$ and $R>0$ be such that $B(x_0,R)\subset\Om$. Consider a sequence $f_n$ in $ L^1(\Om)^N$ and a function $f$ in $L^1(\Om)^N$ such that $\div f_n$ converges weakly to $\div f$ in $W^{-1,1}(\Om)$, and define $h_n$ in $(0,R)$ by
\beq\label{hn}
h_n(r):=\int_{\partial B(x_0,r)}f_n\cdot{x-x_0\over|x-x_0|}\,ds,\quad\mbox{for }r\in (0,R).
\eeq
Then, the sequence $h_n$ is bounded and equi-integrable in $L^1(0,R)$.
\end{Lem}
\begin{proof}
It is equivalent to prove that  $h_n$ converges weakly in $L^1(0,R)$. For this purpose, consider $\phi\in L^\infty(0,R)$, and define $\varphi\in W^{1,\infty}(0,R)$ with $\varphi(R)=0$, by
\[
\varphi(r)=\int_r^R \phi(t)\,dt\quad\mbox{for }r\in[0,R].
\]
Then, we have
\[
\ba{l}\dis\int_0^Rh_n\phi\,dr=-\int_{B(x_0,R)}f_n \cdot{x-x_0\over|x-x_0|}\,\varphi'(|x-x_0|)\,dx=
-\int_{B(x_0,R)}f_n \cdot\nabla \big[\varphi(|x-x_0|)\big]dx
\\ \ecart
\dis =\big\langle\div f_n,\ph(|x-x_0|)\big\rangle\limi \big\langle\div f,\ph(|x-x_0|)\big\rangle=\int_0^R h\phi\,dr,
\ea
\]
where $h\in L^1(0,R)$ is defined replacing $f_n$ by $f$ in formula (\ref{hn}). Therefore, $h_n$ converges weakly to $h$ in $L^1(0,R)$.
\end{proof}
\par\noindent
{\bf Proof of Theorem~\ref{thmdc}.}
First of all, if $\sigma$ and $\eta$ satisfy the regularity assumption \eqref{siLsetaLs'}, then the weak formulations \eqref{concte1} and \eqref{concte3} are reduced to $\mu=\sigma:\eta$. Indeed, in this case any function $u$ satisfying \eqref{eta-DuLp'} or \eqref{eta-DuW1N} is in $W^{1,s'}(\Om)^N$, so that
\[
\div\,(\sigma^T u)=\Div\,(\sigma)\cdot u+\sigma:Du.
\]
A simple integration by parts in \eqref{concte1} and \eqref{concte3} then yields $\mu=\sigma:\eta$. 
\par
Let us now treat the general case.
From Proposition~2.5 of \cite{BCM} we deduce the existence of functions $u_n,u$ in $W^{1,q}(\Om)^N$ satisfying
\beq\label{cvun}
u_n\rightharpoonup u\quad\hbox{in }
\left\{\ba{ll}
W^{1,q}(\Om)^M, & \mbox{if }q>1
\\ \ecart
BV(\Om)^M, & \mbox{if }q=1,
\ea\right.
\eeq
\beq\label{cvetan-Dun}
\eta_n-Du_n\to\eta-Du\quad\hbox{strongly in }
\left\{\ba{ll}
L^{p'}_{\rm loc}(\Om)^{M\times N}, & \hbox{if }p>1
\\ \ecart
W^{1,N}_{\rm loc}(\Om)^{M\times N}, & \hbox{if }p=1.
\ea\right.
\eeq
\par
Let be a closed ball of radius $R>0$ contained in $\Om$. Up to a translation we may assume the ball is centered at the origin.
Define $v_n,v:(0,R)\times S_{N-1}\to\RR^M$ by \eqref{devn}. 
For $R_0\in (0,R)$ and for a closed set $U$ of $[R_0,R]$ such that $v\in C^0(U;W^{1,q}(S_{N-1}))^M$, take a function $\varphi\in W^{1,\infty}(0,\infty)$ with ${\rm supp}\,\varphi\subset [0,R]$, ${\rm supp}\,(\varphi')\subset U$.
Then, for a fixed $\lambda>0$, consider the set $U_n$ defined by \eqref{l2pr2Un}
and define the function $\varphi_n\in W^{1,\infty}(0,\infty)$ by
\[
\varphi_n(r):=\int_R^r \varphi'\,1_{U_n}\,ds,\quad\mbox{for }r\in [0,\infty).
\]
Also define the functions $\psi_n,\psi\in W^{1,\infty}_0(\Om)$ by
\[
\psi_n(x):=\varphi_n(|x|),\quad \psi(x):=\varphi(|x|),\quad\mbox{for }x\in \Om.
\]
\par
According to Remark~\ref{rem3tp} our aim is to pass to the limit in  $\big\langle \sigma_n:\eta_n,\psi_n\big\rangle$.
We distinguish three cases:
\par\medskip\noindent
$\bullet$ {\it The case $p,q> 1$.} Using assumption \eqref{siLsetaLs'} or the more general \eqref{concte2}, combined with the first convergences of \eqref{dvnqcwnp} and \eqref{cvetan-Dun}, we have
\beq\label{dthp4}
\ba{l}
\dis \big\langle \sigma_n:\eta_n,\psi_n\big\rangle
\\ \ecart
\dis =-\,\big\langle{\rm Div}\,\sigma_n, u_n\psi_n\big\rangle
+\into\sigma_n:(\eta_n-Du_n)\,\psi_n\,dx-\int_{\{|x|\in U_n\}}(\sigma_n\nabla\psi_n)\cdot u_n\,dx
\\ \ecart
\dis =-\,\big\langle{\rm Div}\,\sigma, u\psi\big\rangle
+\into\sigma:(\eta-Du)\,\psi\,dx-\int_{\{|x|\in U_n\}}(\sigma_n\nabla\psi_n)\cdot u_n\,dx+o(1).
\ea
\eeq
On the one hand, to pass to the limit in the left-hand side of \eqref{dthp4} we use  the decomposition
\[
\big\langle\sigma_n:\eta_n,\psi_n\big\rangle=\big\langle\sigma_n:\eta_n,\psi\big\rangle+\big\langle\sigma_n:\eta_n,\psi_n-\psi\big\rangle 
\]
where  the first term converges clearly to $\langle\mu,\psi\rangle$ by \eqref{cvsinetanmu}.
For the second term, by \eqref{cvsinetanmu} there exist functions $f_n\in L^1(\Om)^N$ satisfying $\div f_n=\sigma_n:\eta_n$.
Thus, we have
\[
\big|\big\langle\sigma_n:\eta_n,\psi_n-\psi\big\rangle\big|=\left|\,\into f_n\cdot\nabla (\psi_n-\psi)\,dx\,\right|\leq C\int_{U\setminus U_n}|h_n|\,dr,
\]
where $h_n$ is defined by (\ref{hn}). Hence, by \eqref{l1pr1Un} we get that
\beq\label{dthp5b}
\limsup_{n\to\infty}\left|\,\big\langle\sigma_n:\eta_n,\psi_n\big\rangle-\langle\mu,\psi\rangle\,\right|
\leq C\sup_{m\in\NN\atop |B|\leq c/\lambda}\int_B|h_m|\,dr.
\eeq
\par
On the other hand, for the last term in \eqref{dthp4}, consider the functions $v_n,v$ of \eqref{devn} and define the functions $\xi_n,\xi:(0,R)\times S_{N-1}\to\RR^{M\times N} $ by
\[
\xi_n(r,y):=\sigma_n(ry),\quad \xi(r,y):=\sigma(ry),\quad\hbox{ a.e. }(r,y)\in (0,R)\times S_{N-1}.
\]
Then, we have
\beq\label{sinxin}
\ba{l}\dis\int_{\{|x|\in U_n\}}(\sigma_n\nabla\psi)\cdot u_n\,dx=\int_{U_n}\varphi'(r)\,r^{N-1}\int_{S_{N-1}}(\xi_ny)\cdot (v_n-v)\,ds(y)\,dr
\\ \ecart
\dis +\int_{U}\varphi'(r)\,r^{N-1}\int_{S_{N-1}}(\xi_ny)\cdot v\,ds(y)\,dr
-\int_{U\setminus U_n}\varphi'(r)\,r^{N-1}\int_{S_{N-1}}(\xi_ny)\cdot v\,ds(y)\,dr.
\ea
\eeq
Since $\xi_n$ is bounded in $L^p(R_0,R;L^p(S_{N-1}))^{M\times N}$ and $v_n$ satisfies the first convergence of \eqref{l2pr3Un} with $X:=L^{p'}(S_{N-1})$ and $p'<q^*_{N-1}$ by \eqref{inepqN-1}, the first term in the right-hand side of \eqref{sinxin} tends to zero.
Moreover, since $\xi_n$ converges weakly to $\xi$ in $L^p(U;L^p(S_{N-1}))^{M\times N}$ and $v$ is in $C^0(U;L^{p'}(S_{N-1}))^M$ by Sobolev's imbedding combined with $p'<q^*_{N-1}$, we have
\[
\ba{ll}
\dis \int_{U}\varphi'(r)\,r^{N-1}\int_{S_{N-1}}(\xi_ny)\cdot v\,ds(y)\,dr
& \dis \to \int_{U}\varphi'(r)\,r^{N-1}\int_{S_{N-1}}(\xi y)\cdot v\,ds(y)\,dr
\\ \ecart
& \dis =\into(\sigma\nabla\psi)\cdot u\,dx.
\ea
\]
The last term of \eqref{sinxin} can be estimated thanks to H\"older's inequality by
\[
\ba{l}
\dis \left|\,\int_{U\setminus U_n}\varphi'(r)\,r^{N-1}\int_{S_{N-1}}(\xi_ny)\cdot v\,ds(y)\,dr\,\right|
\\ \ecart
\dis \leq R^{N-1}\,|U\setminus U_n|^{1\over p'}\left\|\varphi'\right\|_{L^\infty(U)}\|\xi_n\|_{L^p(U;L^p(S_{N-1}))^{M\times N}}\,\|v\|_{C^0(U;L^{p'}(S_{N-1}))^M},
\ea
\]
hence by \eqref{l1pr1Un}
\beq\label{dthp8}
\left|\,\int_{U\setminus U_n}\varphi'(r)\,r^{N-1}\int_{S_{N-1}}(\xi_ny)\cdot v\,ds(y)\,dr\,\right|\leq{C\over \lambda^{1\over p'}}.
\eeq
Finally, combining \eqref{dthp4}, \eqref{dthp5b}, \eqref{dthp8} we obtain
\[
\ba{l}
\dis \left|\,\langle\mu,\psi\rangle+\big\langle{\rm Div}\,\sigma, u\psi\big\rangle-\into\sigma:\big[\eta\psi-D(u\psi)\,dx\,\right|
\\ \ecart
\dis =\left|\,\langle\mu,\psi\rangle+\big\langle{\rm Div}\,\sigma, u\psi\big\rangle-\into\sigma:(\eta-Du)\,\psi\,dx+\into(\sigma\nabla\psi)\cdot u\,dx\,\right|
\\ \ecart
\dis \leq C\left(\sup_{m\in\NN\atop |B|\leq c/\lambda}\int_B|h_m|\,dr+{1\over \lambda^{1\over p'}}\right).
\ea
\]
Taking into account the equi-integrability of $h_m$ given by Lemma \ref{lem.wcW-11}
and the arbitrariness of $\lambda>0$, we have just proved that the function $u$ defined by \eqref{cvetan-Dun} satisfies \eqref{concte1} for any $\varphi\in W^{1,\infty}(0,\infty)$ with ${\rm supp}\,\varphi\subset [0,R]$, and ${\rm supp}\,(\ph')$ contained in a closed set $U$ of $[R_0,R]$ such that $v$ belongs to $C^0(U;W^{1,q}(S_{N-1}))^M$.
\par
Finally, by Lusin's theorem the closed set $U$ of $(0,R]$ can be chosen such that $R-|U|$ is arbitrary small. Hence, any function $\varphi\in W^{1,\infty}(0,\infty)$, with ${\rm supp}\,\varphi\subset[0,R]$, can be approximated for the weak-$*$ topology of $W^{1,\infty}(0,\infty)$ by a sequence of functions
\[
\varphi_U(r):=\int_R^r \varphi'\,1_U\,ds,\quad\mbox{for }r\in[0,\infty),
\]
which satisfy ${\rm supp}\,(\varphi_U)\subset[0,R]$ and ${\rm supp}\,(\varphi'_U)\subset U$. This combined with the density argument of Remark~\ref{rem.nregsinetan} (based on the fact that $\mu\in W^{-1,1}(\Om)$) shows that the weak formulation \eqref{concte1} holds actually for any $\varphi\in W^{1,\infty}(0,\infty)$, with ${\rm supp}\,\varphi\subset[0,R]$.
This concludes the proof of Theorem~\ref{thmdc} in the case $p,q>1$.
\par\medskip\noindent
$\bullet$ {\it The case $q=1$.} It is similar to the previous case using the first convergence of \eqref{cvetan-Dun}, and the second convergence of~\eqref{cvun} combined with Sobolev's imbedding $BV(\Om)^M\hookrightarrow L^{N'}(\Om)^M$.
\par\medskip\noindent
$\bullet$ {\it The case $p=1$.} It is also similar to the first case. The only delicate point comes from the second term in the right-hand side of \eqref{dthp4}.
In view of \eqref{si-zeMLq'} and \eqref{si-zeDu} we can write
\beq\label{sinzenDun}
\into\sigma_n:(\eta_n-Du_n)\,\psi_n\,dx=\into(\sigma_n-\zeta_n):(\eta_n-Du_n)\,\psi_n\,dx+\into\zeta_n:(\eta_n-Du_n)\,\psi_n\,dx,
\eeq
where by virtue of Proposition~2.5 of \cite{BCM} the measures $\zeta_n,\zeta$ satisfy
\beq\label{cvsinzen}
\zeta_n\rightharpoonup\zeta\quad\hbox{in }\M(\Om)^{M\times N},\quad{\rm Div}\,\zeta_n=0\;\;\mbox{in }\Om,
\quad\sigma_n-\zeta_n\to\sigma-\zeta\;\;\mbox{strongly in }L^{q'}_{\rm loc}(\Om)^{M\times N}.
\eeq
By the second convergence of \eqref{cvetan-Dun} and \eqref{cvsinzen} the first term in the right-hand side of \eqref{sinzenDun} clearly converges. Moreover, we can also pass to the limit in the second term of the right-hand side of \eqref{sinzenDun}, since the divergence free sequence $\zeta_n$ converges weakly in $W^{-1,N'}(\Om)^{M\times N}$ thanks to the Bourgain, Brezis result \cite{BoBr}, hence
\[
\into\sigma_n:(\eta_n-Du_n)\,\psi_n\,dx\to\into\sigma:(\eta-Du)\,\psi\,dx.
\]
Therefore, the proof of Theorem~\ref{thmdc} is complete.
\cqfd
\subsection{The limit case: ${1\over p}+{1\over q}=1+{1\over N-1}$}
When inequality \eqref{inepqN-1} becomes an equality, the imbedding $W^{1,q}(S_{N-1})\hookrightarrow L^{p'}(S_{N-1})$ is no more compact, so Lemma~\ref{lemLpW1q} is useless. This lack of compactness can be overcome adding an equi-integrability assumption for the sequence $\eta_n$ in Theorem~\ref{thmdc}. This is the aim of Theorem~\ref{calipn1} in the case $p>1$.
\par
The case $p=1$, and thus $q=N-1$, corresponds to the critical case for Sobolev's inequality: $W^{1,N-1}(S_{N-1})$ is continuously imbedded in $L^s(S_{N-1})$ for any $s\in [1,\infty)$, but if $N>2$, it is not imbedded in $L^\infty(S_{N-1})$. To get over this difficulty we need to work with a space which is a little more regular than $L^{N-1}(\Om)$.
So, in Theorem~\ref{calipi1} below $L^{N-1}(\Om)$ is replaced by the Lorentz space $L^{N-1,1}(\Om)$. It is known that the space of functions $u\in W^{1,N-1}(S_{N-1})$ the gradient of which belongs to $L^{N-1,1}(S_{N-1})$ is continuously imbedded in $C^0(S_{N-1})$ (see, {\em e.g.}, \cite{Tar3}, Chap.~31). Moreover, for $N=2$, $L^{1,1}(S_1)$ agrees with $L^1(S_1)$, so that we can extend Theorem~\ref{thmdc} to the case $N=2$, $p=q=1$.
\begin{Thm}\label{calipn1}
Let $\Om$ be a bounded open set of $\RR^N$, with $N\geq 2$, and let $p,q$ be such that
\beq\label{pqN}
1<p\leq N-1, \quad 1\leq q<N-1,\quad {1\over p}+{1\over q}=1+{1\over N-1}.
\eeq
Consider two sequences of matrix-valued functions $\sigma_n\in L^p(\Om)^{M\times N}$, $\eta_n\in L^q(\Om)^{M\times N}$,
satisfying  \eqref{siLsetaLs'}, \eqref{cvsinetanmu}, \eqref{cvvnpwnq}, \eqref{dvnqcwnp} together with
\beq\label{equico}
|\eta_n|^p\hbox{ equi-integrable in }L^1(\Om).
\eeq
Then the weak formulation \eqref{concte1} holds true.
\end{Thm}
In order to state the case $p=1$, $q=N-1$, recall the definition of the Lorentz space $L^{p,1}(E)$:
\begin{Def} Let $E$ be a measurable set of $\RR^N$. For a measurable function $f:E\to\RR$, the non-increasing rearrangement $f^\ast:[0,\infty)\to\RR$ of $f$ is defined by
\beq\label{defast}
f^\ast(t):=\inf\big\{\lambda\geq 0:\big|\{x\in E:|f(x)|>\lambda\}\big|\leq t\big\}.
\eeq
Then, we define $L^{p,1}(E)$, $p>1$, as the space of measurable functions $f:E\to\RR$ such that
\beq\label{nolor}
\|f\|_{L^{p,1}(E)}=\int_0^\infty t^{-{1\over p'}}f^\ast(t)\,dt=\int_0^\infty \big|\{x\in E:|f(x)|>\lambda\}\big|^{1\over p}\,d\lambda<\infty.
\eeq
The space $L^{p,1}(E)$ is a Banach space equipped with the norm $\|\cdot\|_{L^{p,1}(E)}$.
\end{Def}
\begin{Thm}\label{calipi1}
Let $\Om$ be a bounded open set of $\RR^N$, with $N\geq 2$, and two sequences of matrix-valued functions $\sigma_n$ and $\eta_n$
satisfying \eqref{siLsetaLs'}, \eqref{cvsinetanmu},
\beq\label{cvvnpwnqcp1}
\left\{\ba{ll}
\sigma_n\stackrel{\ast}\rightharpoonup \sigma & \mbox{ in }\M(\Om)^{M\times N}
\\ \ecart
\eta_n\rightharpoonup \eta & \mbox{ in }L^{N-1,1}(\Om)^{M\times N},
\ea\right.
\eeq
\beq\label{dvnqcwnpcp1}
\left\{\ba{ll}
\Div\,\sigma_n\to \Div\,\sigma & \mbox{ in }W^{-1,(N-1)'}(\Om)^{M}
\\ \ecart
\Curl\,\eta_n\to \Curl\,\eta & \mbox{ in }L^N(\Om)^{M\times N\times N}.
\ea\right.
\eeq
Also assume that the sequence $\eta_n$ satisfies the equi-integrability condition
\beq\label{equllo}
\forall\,\ep>0,\ \exists\,\delta>0,\quad\|\eta_n\|_{L^{N-1,1}(E)^{M\times N}}\leq\ep,\;\;\forall\,n\in\NN,\ \forall\,E\mbox{ measurable set of }\Om,\ |E|<\delta.
\eeq
Then, for any function $u$ satisfying
\beq\label{eta-DuLN-11}
u\in W^{1,N-1}(\Om)^M,\quad Du\in L^{N-1,1}(\Om)^{M\times N},\quad \eta-Du\in W^{1,N}(\Om)^{M\times N},
\eeq
the limit $\mu$ satisfies the weak formulation
\beq\label{concte4}
\left\{\ba{l}\dis\forall\,B(x_0,R)\Subset\Om,\ \forall\,\varphi\in W^{1,\infty}(0,\infty),\mbox{ with }{\rm supp}\,\varphi\subset [0,R],\mbox{ such that }
\\ \ecart\dis
\exists\,U\mbox{closed set of }[0,R],\mbox{ with }u(x_0+ry)\in C^0(U;X^{1,N-1}(S_{N-1}))^M,\ {\rm supp}\,(\ph')\subset U,
\\ \ecart
\dis \langle\mu,\psi\rangle=-\,\big\langle{\rm Div}\,\sigma, u\psi\big\rangle
+\int_{B(x_0,R)}\big[\sigma(dx):(\eta-Du)\,\psi-(\sigma\nabla\psi)\cdot u\,dx\big],
\\ \ecart
\dis\hbox{where }\psi(x):=\varphi(|x-x_0|),
\ea\right.
\eeq
and $X^{1,N-1}(S_{N-1})$ is the space defined by
\beq\label{X1N-1}
X^{1,N-1}(S_{N-1}):=\left\{v\in W^{1,N-1}(S_{N-1}): \nabla v\in L^{N-1,1}(S_{N-1})^N\right\}.
\eeq
\end{Thm}
\begin{Rem}\label{rem.dcLN-11}
Let $u$ be a function in $W^{1,N-1}(\Om)^M$ such that $Du\in L^{N-1,1}(\Om)^{M\times N}$, and let $v:(0,R)\times S_{N-1}\to\RR^M$ be the function defined by $v(r,y):=u(x_0+ry)$, so that $\nabla_y v$ is the tangential part of $\nabla u$ on $\partial B(x_0,r)$.
By H\"older's inequality we have for any $\lambda>0$,
\[
\ba{ll}
\dis \int_0^R r^{N-1}\left(\int_{S_{N-1}}\kern -.4cm 1_{\{|\nabla_y v|>\lambda\}}\,ds(y)\right)^{1\over N-1}\kern -.2cm dr
& \kern -.2cm \dis \leq C R^{N'(N-2)}\left(\int_0^R\kern -.2cm\int_{S_{N-1}}\kern -.4cm 1_{\{|\nabla u(x_0+ry)|>\lambda\}}\,ds(y)\,r^{N-1}\,dr\right)^{1\over N-1}
\\ \ecart
& \kern -.2cm \dis \leq C R^{N'(N-2)}\Big(\left|\big\{x\in B(x_0,R):|\nabla u(x)|>\lambda\big\}\right|\Big)^{1\over N-1}.
\ea
\]
Hence, integrating the previous inequality with respect to $\lambda>0$ and using that $\nabla u\in L^{N-1,1}(\Om)^N$, it follows that $v$ is in $L^1_{r^{N-1}dr}(0,R;X^{1,N-1}(S_{N-1}))^M$, and thus in $L^1_{r^{N-1}dr}(0,R;C^0(S_{N-1}))^M$ since the Lorentz space $L^{N-1,1}(S_{N-1})$ is imbedded into $C^0(S_{N-1})$ (see \cite{Tar3}, Chap.~31). Moreover, by Lusin's theorem, for any $\ep>0$, there exists a closed set $U$ satisfying the second line of \eqref{concte4} such that $|U|>R-\ep$.
Hence, for $\sigma\in\M(\Om)^{M\times N}$,  the integral term
\[
\int_{B(x_0,R)}(\sigma\nabla\psi)\cdot u\,dx,\quad\mbox{where }\psi(x):=\varphi(|x-x_0|),
\]
has a sense for any function $\ph$ satisfying the two first lines of \eqref{concte4}. Therefore, we can conclude as in Remark~\ref{rem.p=1} that the weak formulation \eqref{concte4} fully characterizes the distribution $\mu$.
\end{Rem}
The proof of the two last theorems is similar to the one of Theorem~\ref{thmdc} using Lemma~\ref{le1cc} below in the case $p>1$, and
Lemma~\ref{le2cc} below in the case $p=1$, instead of Lemma~\ref{lemLpW1q}. So we restrict ourselves to the proof of these two lemmas.
\begin{Lem}\label{le1cc}
Let $N>2$, let $R_0,R>0$ be such that $R_0<R$, and let $q\in[1,N-1)$. Consider a sequence $u_n$ in $W^{1,q}(C(R_0,R))$ which converges weakly to a function $u$ in $W^{1,q}(C(R_0,R))$, and such that $|\nabla u_n|^q$ is equi-integrable in $L^1(\Om)$.
Consider $v_n,v\in L^q(R_0,R;W^{1,q}(S_{N-1}))$ defined by~\eqref{devn}.
\\
Then, for any $U$ subset of $[R_0,R]$ such that $v\in L^\infty(U; L^{q^*_{N-1}}(S_{N-1}))^M$, for any $\lambda,\ep>0$, there exists a sequence $U_n\subset U$ satisfying
\beq\label{pr1Un}
|U\setminus U_n|\leq 
{1\over R_0^{N-1}}\left({1\over \lambda}\int_{\{|x|\in U\}}\big(|\nabla u_n|^q+|\nabla u|^q\big)\,dx+\ep\right),
\eeq
\beq\label{pr2Un}
\int_{S_{N-1}}\big(|\nabla u_n(ry)|^q+|\nabla u(ry)|^q\big)\,ds(y)<\lambda,\quad \;\;\hbox{a.e. } r\in U_n,
\eeq
\beq\label{pr3Un}
\|v_n-v\|_{L^\infty(U_n;L^{q^*_{N-1}}(S_{N-1}))}\to 0.
\eeq
\end{Lem}\noindent
\begin{proof} 
Since $W^{1,p}(C(R_0,R))$ is compactly imbedded in $L^1(\partial B(0,r))$ for any $r\in [R_0,R]$, the sequence $v_n(r,.)$ converges to $v(r,.)$ in $L^1(S_{N-1})^M$ for any $r\in [R_0,R]$. Also using that 
\[
\|v_n(r_1,\cdot)-v_n(r_2,\cdot)\|_{L^1(S_{N-1})}\leq C\int_{\{r_1<|x|<r_2\}}|\nabla u_n|\,dx,\quad\forall\,r_1,r_2\hbox{ with }R_0<r_1<r_2<R.
\]
and the equi-integrability of $|\nabla u_n|$ in $L^1(R_0,R)$, we easily conclude that $v_n$ converges to $v$ in $C^0([R_0,R];L^1(S_{N-1}))^M$.\par
Now, take $\ep>0$. By the equi-integrability of $|D u_n|^q$, for any $k\in \NN$, there exists $\delta_k>0$ such that for any measurable set $B\subset C(R_0,R)$ with $|B|<\delta_k$, we have
\beq\label{equile}
\int_{B}\varLambda_n\,dx<{\ep^2\over 2^{2k}},\quad\forall\, n\in\NN,\qquad\mbox{where }\varLambda_n:=|\nabla u_n|^q+|\nabla u|^q.
\eeq
Let $\phi:(0,\infty)\to\RR$ the function defined by 
\beq\label{promukb}
\phi(h):=\big|B(e_1,h)\cap S_{N-1}\big|,\quad\mbox{for }h>0,
\eeq
and let $h_k>0$ be such that
\beq\label{promuk}
\phi(h_k)\,{R^N-R_0^N\over N}<\delta_k.
\eeq
Then, for a.e. $r\in (R_0,R)$ and any $n,k\in\NN$, denote
\[
T_{n,k}(r):=\sup_{z\in S_{N-1}}\int_{B(z,h_k)\cap S_{N-1}}\hskip-30pt\varLambda_n(ry)\,ds(y).
\]
\par
We will prove that the set
\beq\label{accopo0}
E_{n,k}:=\left\{r\in (R_0,R): T_{n,k}(r)>{\ep\over 2^k}\right\},\quad\mbox{for }k,n\in\NN,
\eeq
satisfies
\beq\label{accopo}
|E_{n,k}|<{\ep\over 2^kR_0^{N-1}},\quad\forall\,k,n\in\NN.
\eeq
To this end, for fixed $k,n\in\NN$, consider for a.e. $r\in (0,R)$,
\[
F(r):=\left\{z\in S_{N-1}:\int_{B(z,h_k)\cap S_{N-1}}\hskip-30pt\varLambda_n(ry)\,ds(y)=\sup_{x\in S_{N-1}}\int_{B(x,h_k)\cap S_{N-1}}\hskip-30pt\varLambda_n(ry)\,ds(y)\right\}.
\]
Then, $F$ is a multifunction valued on closed sets. Let us also prove that it is measurable, {\em i.e.} that for any open set $G\subset S_{N-1}$ we have
\beq\label{deGmea}
\dis\big\{r\in (R_0,R): F(r)\cap G\not={\rm\O}\big\}\;\;\hbox{is measurable}.
\eeq
For this purpose, consider a sequence of points $z_l\in S_{N-1}$, which is dense in $S_{N-1}$. Then, taking into account that
\[
\sup_{z\in S_{N-1}}\int_{B(z,h_k)\cap S_{N-1}}\hskip-30pt\varLambda_n(ry)\,ds(y)=\sup_{l\in\NN}\int_{B(z_l,h_k)\cap S_{N-1}}\hskip-30pt\varLambda_n(ry)\,ds(y),
\]
we deduce that $T_{n,k}$ is measurable.
Now, consider an open set $G\subset S_{N-1}$ and observe that by continuity, $F(r)\cap G$ is not empty if and only if 
\[
\forall\, m\in\NN,\ \exists z_l\in G,\quad T_{n,k}(r)-\int_{B(z_l,h_k)\cap S_{N-1}}\hskip-30pt\varLambda_n(ry)\,ds(y)<{1\over m}.
\]
 Therefore, we get that
 \[
 \big\{r\in (R_0,R): F(r)\cap G\neq{\rm\O}\big\}=\bigcap_{m\in\NN}\,\bigcup_{z_l\in G}\,\left\{r\in(R_0,R):
T_{n,k}(r)-\int_{B(z_l,h_k)\cap S_{N-1}}\hskip-30pt\varLambda_n(ry)\,ds(y)<{1\over m}\right\}
 \]
 is measurable.
 \par 
 Since $F$ is valued on closed sets and measurable, we can apply the selection measurable theorem to derive a measurable function $g:(R_0,R)\mapsto S_{N-1}$ satisfying
 \[
 \int_{B(g(r),h_k)\cap S_{N-1}}\hskip-30pt\varLambda_n(ry)\,ds(y)=\sup_{z\in S_{N-1}}\int_{B(z,h_k)\cap S_{N-1}}\hskip-30pt\varLambda_n(ry)\,ds(y),\;\;\hbox{a.e. }r\in (R_0,R).
 \]
This implies that
\beq\label{defcB}
B:=\left\{ry\in C(R_0,R):y\in B\big(g(r),h_k\big)\cap S_{N-1}\right\}
\eeq
is a measurable set of $C(R_0,R)$, which by \eqref{promukb}, \eqref{promuk} and rotation invariance satisfies
\beq\label{medcB}
|B|\leq\phi(h_k)\int_{R_0}^Rr^{N-1}\,dr=\phi(h_k)\,{R^N-R_0^N\over N}<\delta_k.
\eeq
By \eqref{equile} and \eqref{accopo0} this yields
\[
{\ep^2\over 2^{2k}}>\int_B\Lambda_n\,dx\geq\int_{R_0}^RT_{n,k}(r)\,r^{N-1}\,dr\geq{\ep\over 2^k}\,R_0^{N-1}\,|E_{n,k}|,
\]
hence the desired estimate \eqref{accopo}.
\par
Now, for $\lambda>0$, define the set
\beq\label{defUn}
U_{n}:=\left\{r\in U\setminus\bigcup_{k\in\NN}E_{n,k}:\int_{S_{N-1}}\varLambda_n(ry)\,ds(y)<\lambda\right\}.
\eeq
Then, \eqref{pr2Un} is satisfied by definition, while
\[
\ba{ll}
|U_n| & \dis \geq \left|\,\left\{r\in U: \int_{S_{N-1}}\varLambda_n(ry)\,ds(y)<\lambda\right\}\,\right|-\sum_{k\in\NN}|E_{n,k}|
\\ \ecart
& \dis \geq |U|- \left|\,\left\{r\in U: \int_{S_{N-1}}\varLambda_n(ry)\,ds(y)\geq\lambda\right\}\,\right|-{\ep\over R_0^{N-1}}
\\ \ecart
& \dis \geq |U|-{1\over R_0^{N-1}}\left({1\over \lambda}\int_{\{|x|\in U\}}\varLambda_n\,dx+\ep\right),
\ea
\]
which gives \eqref{pr1Un}.
\par
Let us now prove that \eqref{pr3Un} holds.
For this purpose, fix $k\in\NN$ and, using Vitali's covering theorem, consider $y_1,\cdots, y_{n_k}\in S_{N-1}$ such that
\[
S_{N-1}\subset \bigcup_{i=1}^{n_k}B(y_{i},h_k),\quad B(y_{i},h_k/5)\cap B(y_{j},h_k/5)={\rm\O},\;\;\hbox{if }i\not=j.
\]
Then, we have
\beq\label{lecfu}\ba{l}\dis\|v_n-v\|^{q^*_{N-1}}_{L^\infty(U_n;L^{q^*_{N-1}}(S_{N-1}))}
\leq\left\|\,\sum_{i=1}^{n_k}\|v_n-v\|^{q^*_{N-1}}_{L^{q^*_{N-1}}(B(y_i,h_k)\cap S_{N-1}))}\,\right\|_{L^\infty(U_n)}
\\ \ecart
\dis \leq 3^{q^*_{N-1}-1}\left\|\,\sum_{i=1}^{n_k}\Big\|\,v_n-{1\over \phi(h_k)}\int_{B(y_i,h_k)\cap S_{N-1}}\hskip -15pt v_n\,ds(y)\,\Big\|_{L^{q^*_{N-1}}(B(y_i,h_k)\cap S_{N-1}))}^{q^*_{N-1}}\,\right\|_{L^\infty(U_n)}
\\ \ecart
\dis +\,{3^{q^*_{N-1}-1}\over \phi(h_k)^{q^*_{N-1}}}\,\sum_{i=1}^{n_k}\left\|\,\int_{B(y_i,h_k)\cap S_{N-1}}\hskip -15pt(v_n-v)\,ds(y)\,\right\|^{q^*_{N-1}}_{L^\infty(U_n)}
\\ \ecart
\dis +\,3^{q^*_{N-1}-1}\left\|\,\sum_{i=1}^{n_k}\Big\|\,v-{1\over \phi(h_k)}\int_{B(y_i,h_k)\cap S_{N-1}}\hskip -15pt v\,ds(y)\,\Big\|_{L^{q^*_{N-1}}(B(y_i,h_k)\cap S_{N-1}))}^{q^*_{N-1}}\,\right\|_{L^\infty(U_n)}.
\ea
\eeq
Using the invariance by dilatations of Sobolev-Wirtinger's inequality it follows from \eqref{defUn} and \eqref{accopo0} that
\[
\ba{l}
\dis \left\|\,\sum_{i=1}^{n_k}\Big\|\,v_n-{1\over \phi(h_k)}\int_{B(y_i,h_k)\cap S_{N-1}}\hskip -15pt v_n\,ds(y)\,\Big\|_{L^{q^*_{N-1}}(B(y_i,h_k)\cap S_{N-1}))}^{q^*_{N-1}}\,\right\|_{L^\infty(U_n)}
\\ \ecart
\dis \leq\left\|\,\sum_{i=1}^{n_k}\left\|\nabla u_n(ry)\right\|_{L^{q}(B(y_i,h_k)\cap S_{N-1}))^N}^{q^*_{N-1}}\,\right\|_{L^\infty(U_n)}
\\ \ecart
\dis \leq C\mathop{\mbox{ess-sup}}_{\quad r\in U_n\atop \quad z\in S_{N-1}}\left(\|\nabla u_n(ry)\|_{L^{q}(B(z,h_k)\cap S_{N-1}))}^{q^*_{N-1}-q}\right)\left\|\nabla u_n\right\|^q_{L^\infty(U_n;L^q(S_{N-1}))}
\\ \ecart
\dis =C\mathop{\mbox{ess-sup}}_{\quad r\in U_n\atop \quad z\in S_{N-1}}\big(T_{n,k}(r)\big)^{q^*_{N-1}-q}\left\|\nabla u_n\right\|^q_{L^\infty(U_n;L^q(S_{N-1}))}\leq C\left({\ep\over 2^k}\right)^{q^*_{N-1}-q}\lambda.
\ea
\]
The same reasoning also shows that
\[
\left\|\,\sum_{i=1}^{n_k}\Big\|\,v-{1\over \phi(h_k)}\int_{B(y_i,h_k)\cap S_{N-1}}\hskip -15pt v\,ds(y)\,\Big\|_{L^{q^*_{N-1}}(B(y_i,h_k)\cap S_{N-1}))}^{q^*_{N-1}}\,\right\|_{L^\infty(U_n)}\leq C\left({\ep\over 2^k}\right)^{q^*_{N-1}-q}\lambda.
\]
Since $v_n$ converges to $v$ in  $L^\infty(R_0,R_1;L^1(S_{N-1}))$, the second term in the right-hand side of \eqref{lecfu} tends to zero. Therefore, taking the limsup as $\ep\to 0$ in \eqref{lecfu} we obtain that
\[
\limsup_{n\to \infty}\|v_n-v\|^{q^*_{N-1}}_{L^\infty(U_n;L^{q^*_{N-1}}(S_{N-1}))}\leq C\left({\ep\over 2^k}\right)^{q^*_{N-1}-q}\lambda,
\quad\forall\, k\in \NN,
\]
which finally yields \eqref{pr3Un}.
\end{proof}

\begin{Lem} \label{le2cc}
Let $N\geq 2$, and let $R_0,R>0$ be such that $R_0<R$. Consider a sequence $u_n$ in $W^{1,N-1}(C(R_0,R))$ which converges weakly to a function $u$ in $W^{1,N-1}(C(R_0,R))$ such that $\nabla u_n$ is bounded in $L^{N-1,1}(C(R_0,R))^{ N}$ and satisfies the equi-integrability condition
\beq\label{equi1n}
\forall\,\ep>0,\ \exists\,\delta>0,\quad\|\nabla u_n\|_{L^{N-1,1}(B)^{M}}\leq \ep,\;\;\forall\,n\in\NN,\ \forall\,B\subset C(R_0,R),\ |B|<\delta.
\eeq
Define $v_n,v\in L^{N-1}(R_0,R;W^{1,N-1}(S_{N-1}))$ by \eqref{devn}.
\\
Then, for any closed set $U$ of $[R_0,R]$ such that $v\in C^0(U;X^{1,N-1}(S_{N-1}))$,  for any $\lambda,\ep>0$, there exists a sequence $U_n\subset U$ satisfying
\beq\label{pr1Un2}
|U\setminus U_n|\leq {(R-R_0)^{N-2\over N-1}\over  R_0}\left({1\over \lambda}\,\|\nabla u_n\|_{L^{N-1,1}(\{|x|\in U\})^{N}}
+{1\over \lambda}\,\|\nabla u\|_{L^{N-1,1}(\{|x|\in U\})^{N}}+\ep\right),
\eeq
\beq\label{pr2Un2}
\big\|\,|\nabla u_n(ry)|+|\nabla u(ry)|\,\big\|_{L^{N-1,1}(S_{N-1})}<\lambda,\quad\hbox{a.e. } r\in U_n,
\eeq
\beq\label{pr3Un2}
\|v_n-v\|_{L^\infty(U_n;C^0(S_{N-1}))}\to 0.
\eeq
\end{Lem}
\noindent
\begin{proof} The proof is quite similar to the one of Lemma~\ref{le1cc}. As before we first note that $v_n$ converges to $v$ in $C^0([R_0,R];L^1(S_{N-1}))$.
\par
Now, take $\ep>0$ and $\delta_k>0$, $k\in\NN$, such that for any measurable set $B\subset C(R_0,R)$ with $|B|<\delta_k$, we have
\beq\label{equilan}
\|\Lambda_n\|_{L^{N-1,1}(B)}<{\ep^2\over 2^{2k}},\quad\forall\, n\in\NN,\qquad\mbox{where }\Lambda_n:=|\nabla u_n|+|\nabla u|.
\eeq
Then, consider $h_k>0$ such that \eqref{promuk} holds, and for $r\in (R_0,R)$, $n,k\in\NN$, denote $T_{n,k}(r)$ by
\[
T_{n,k}(r):=\mathop{\mbox{ess-sup}}_{z\in S_{N-1}}\|\Lambda_n(ry)\|_{L^{N-1,1}(B(z,h_k)\cap S_{N-1})}.
\]
Now, the problem is to estimate the measure of the set $E_{n,k}$ defined by
\beq\label{accopo02}
E_{n,k}:=\left\{r\in (R_0,R): T_{n,k}(r)>{\ep\over 2^k}\right\},\quad\mbox{for }k,n\in\NN.
\eeq
For this purpose, proceeding as in the proof of Lemma~\ref{le1cc} we can construct a measurable function $g:(R_0,R)\to S_{N-1}$ such that for a.e. $r\in (0,R)$,
\[
\|\Lambda_n(ry)\|_{L^{N-1,1}(B(g(r),h_k)\cap S_{N-1})}=
\mathop{\mbox{ess-sup}}_{z\in S_{N-1}}\|\Lambda_n(ry)\|_{L^{N-1,1}(B(z,h_k)\cap S_{N-1})}=T_{n,k}(r).
\]
The set $B$ defined by \eqref{defcB} has a measure less than $\delta_k$. Hence, using successively \eqref{equilan}, H\"older's inequality and \eqref{accopo02} if follows that
\[
\ba{l}
\dis {\ep^2\over 2^{2k}}\geq\|\Lambda_n\|_{L^{N-1,1}(B)}
=\int_0^\infty\left(\int_{R_0}^Rr^{N-1}\int_{B(g(r),h_k)\cap S_{N-1}}\hskip-20pt 1_{\{\Lambda_n>\lambda\}}\,ds(y)\,dr\right)^{1\over N-1}d\lambda
\\ \ecart
\dis \geq {R_0\over (R-R_0)^{N-2\over N-1}}\int_{R_0}^R\int_0^\infty\left(\int_{B(g(r),h_k)\cap S_{N-1}}\hskip-20pt 1_{\{\Lambda_n>\lambda\}}\,ds(y)\right)^{1\over N-1}\hskip -5ptd\lambda\,dr
\\ \ecart
\dis = {R_0\over (R-R_0)^{N-2\over N-1}}\int_{R_0}^R\|\Lambda_n\|_{L^{N-1,1}(B(g(r),h_k)\cap S_{N-1})}\,dr
\geq {R_0\over (R-R_0)^{N-2\over N-1}}\,{\ep\over 2^k}\,|E_{n,k}|,
\ea
\]
which implies that
\beq\label{accopo2}
|E_{n,k}|<{(R-R_0)^{N-2\over N-1}\over R_0}\,{\ep\over 2^k},\quad\forall\,k,n\in\NN.
\eeq
A similar reasoning also shows that
\beq\label{LaLN-11}
\left|\left\{r\in U: \|\Lambda_n(ry)\|_{L^{N-1,1}(S_{N-1})}\geq\lambda\right\}\right|\leq {(R-R_0)^{N-2\over N-1}\over \lambda R_0}\,\|\Lambda_n\|_{L^{N-1,1}(\{|x|\in U\})},\quad\forall\,\lambda>0.
\eeq
Then, defining for $\lambda>0$, the set
\[
U_{n}:=\left\{r\in U\setminus \bigcup_{k\in\NN}E_{n,k}: \|\Lambda_n(ry)\|_{L^{N-1,1}(S_{N-1})}<\lambda\right\},
\]
we deduce from \eqref{accopo2} and \eqref{LaLN-11} that \eqref{pr1Un2} and \eqref{pr2Un2} hold. The proof of \eqref{pr3Un2} is similar to the one of \eqref{pr3Un} taking into account that the space $X^{1,N-1}(S_{N-1})$ defined by \eqref{X1N-1} is continuously imbedded in $C^0(S_{N-1})$.
\end{proof}
\subsection{A counterexample}
In the previous section we have needed some equi-integrability condition to extend the div-curl result of Theorem~\ref{thmdc}
to the case 
\[
{1\over p}+{1\over q}=1+{1\over N-1}.
\]
Actually, the following counterexample  shows that the conclusion of Theorem~\ref{thmdc} is violated in general if the sequences $\sigma_n$ and $\eta_n$ are only bounded in $L^p(\Om)^{M\times N}$ and $L^q(\Om)^{M\times N}$ with \eqref{pqN}.
\par\medskip
Let $N\geq 2$ and $p,q\geq 1$ be such that \eqref{pqN} holds.
Let $\Om:=B'_1\times(0,1)$, where $B'_1$ the unit ball of $\RR^{N-1}$ centered at the origin. The points of $\Om$ are denoted by $(x',x_N)$. We also denote by $x'$ a point of $\Om$ whose last coordinate is zero. Consider the functions $\sigma_n$ and $\eta_n$, $n\geq 1$, defined in cylindrical coordinates by
\beq\label{vnwn3pq}
\left\{\ba{l}
\sigma_n(x):=n^{N-1\over p}\,a_n(|x'|)\,e_N
\\ \ecart\dis
\eta_n(x):=n^{N-1\over p'}\left(a_n'(|x'|)\,x_N\,{x'\over |x'|}+a_n(|x'|)\,e_N\right),
\ea\right.
\ \mbox{where}\quad a_n(r):=(1-r)^n.
\eeq
Then, we have
\begin{Pro}\label{procex}
The sequences $\sigma_n$ and $\eta_n$ defined by \eqref{vnwn3pq} satisfy
\beq\label{dicur}
{\rm div}\,\sigma_n=0,\quad {\rm curl}\,\eta_n=0\quad\hbox{in }\Om,
\eeq
\beq\label{cocont}
\left\{\ba{lll}
\sigma_n\rightharpoonup 0 & \hbox{ in }L^p(\Om)^N, & \hbox{if }p>1
\\ \ecart\dis
\sigma_n\stackrel{\ast}\rightharpoonup |S_{N-2}|\,(N-2)!\,\big(\delta_{\{x'=0\}}\otimes 1\big) & \hbox{ in }\M(\Om)^N, & \mbox{if }p=1,
\ea\right.
\eeq
\beq\label{cocont2}
\left\{\ba{ll}
\eta_n\rightharpoonup 0\;\;\hbox{in }L^q(\Om)^N, &\hbox{ if }q>1
\\ \ecart
\dis \eta_n\stackrel{\ast}\rightharpoonup 0\;\;\hbox{in }\M(\Om)^N, &\hbox{ if }q=1.
\ea\right.
\eeq
while
\beq\label{cocont3}
\sigma_n\cdot\eta_n\stackrel{\ast}\rightharpoonup |S_{N-2}|\,{(N-2)!\over 2^{N-1}}\,\big(\delta_{\{x'=0\}}\otimes 1\big)\;\;\hbox{in }\M(\Om)^N
\quad(\mbox{with }|S_0|:=1).
\eeq
\end{Pro}
\begin{proof} It is clear that $\sigma_n$ is divergence free and $\eta_n$ is curl free in $\Om$. Moreover, a lengthy but easy computation shows that convergences \eqref{cocont}, \eqref{cocont2}, \eqref{cocont3} are a simple consequence of
\beq\label{co1-rn}
(1-|x'|)^n\to 0,\quad\forall\,x'\in\RR^{N-1},\hbox{ with }0<|x'|<1,
\eeq
\beq\label{paliej}
n^{k+1}\int_0^1 r^k(1-r)^{n\alpha}\,dr\to {k!\over \alpha^{k+1}},\quad\forall\, k\in\NN,\ \forall\,\alpha\geq 0.
\eeq
\end{proof}
\section{Applications}
\subsection{Homogenization of systems with non equi-bounded coefficients}
\subsubsection{A $\bfm\Ga$-convergence approach}
Let $\Om$ be a bounded open set of $\RR^N$, $N\geq 2$.
Consider a symmetric non-negative tensor-valued function $\AA_n$ in $L^\infty(\Om)^{(M\times N)^2}$ which satisfies the two following conditions:
\begin{itemize}
\item There exists a constant $\al>0$ such that
\beq\label{ellAal}
\al\int_\Om\,|Dv|^2\,dx\leq\int_\Om\AA_n Dv:Dv\,dx,\quad\forall\,v\in H^1_0(\Om)^M.
\eeq
\item There exists a non-negative Radon measure $\Lambda$ on $\Om$ satisfying
\beq\label{cvAnLa}
\left\{\ba{rlll}
|\AA_n|\stackrel{\ast}\rightharpoonup\La & \mbox{ in }\M(\Om), & & \mbox{if }N=2
\\ \ecart
|\AA_n|^{\rho}\stackrel{\ast}\rightharpoonup \La & \mbox{ in }\M(\Om), & \mbox{ with }\rho>{N-1\over 2}, & \mbox{if }N>2.
\ea\right.
\eeq
\end{itemize}
\par\medskip\noindent
Consider the quadratic functional $F_n$ defined in $L^2(\Om)^M$ by
\beq\label{FnAn}
F_n(v):=\left\{\ba{cl}
\dis \int_\Om\AA_n Dv:Dv\,dx, & \mbox{if }v\in H^1_0(\Om)^M
\\ \ecart
\infty, & \mbox{if }v\in L^2(\Om)^M\setminus H^1_0(\Om)^M.
\ea\right.
\eeq
By a classical compactness result of $\Ga$-convergence (see, {\em e.g.}, \cite{Dal,Bra}) there exist a subsequence of $n$, still denoted by $n$, and a quadratic functional $F:L^2(\Om)^M\to[0,\infty]$ such that $F_n$ $\Ga$-converges to $F$ for the strong topology of $L^2(\Om)^M$, namely for any $v\in L^2(\Om)^M$,
\beq\label{Gacv}
\left\{\ba{ll}
\forall\,v_n\to v\;\;\mbox{ in }L^2(\Om)^M, & \dis F(v)\leq\liminf_{n\to\infty}F_n(v_n),
\\ \ecart
\exists\,\bar{v}_n\to v\;\;\mbox{ in }L^2(\Om)^M, & \dis F(v)=\lim_{n\to\infty}F_n(\bar{v}_n).
\ea\right.
\eeq
Since $F$ is quadratic, it has a bilinear form associated $\Psi:D(F)\times D(F)\to \RR$. We recall that $D(F)$ is a Hilbert space endowed with the scalar product defined by $\Psi$.
\par
Any sequence $\bar{v}_n$ satisfying the second statement of \eqref{Gacv} is called a {\em recovery} sequence for $F_n$ of limit $v$. 
Moreover, let $\bar{v}_n$ be a sequence in $L^2(\Om)^M$ satisfying
\beq\label{bvn}
\left\{\ba{l}
\bar{v}_n\to v\mbox{ strongly in }L^2(\Om)^M
\\ \ecart
\dis \sup_{n\geq 0}F_n(\bar{v}_n)<\infty.
\ea\right.
\eeq
If $\bar{v}_n$ is a recovery sequence for $F_n$ of limit $v$, then 
\beq\label{recseqa}
\lim_{n\to\infty}\int_\Om\AA_nD\bar{v}_n:Dw_n\,dx=\Psi(v,w),\qquad\forall\,w_n\in L^2(\Om)^M,\quad
\left\{\ba{l}
w_n\to w\mbox{  in }L^2(\Om)^M
\\ \ecart
\dis \sup_{n\geq 0}F_n(w_n)<\infty.
\ea\right.
\eeq
Reciprocally, if $\bar v_n$ satisfies
\beq\label{recseq}
\lim_{n\to\infty}\int_\Om\AA_nD\bar{v}_n:Dw_n\,dx=0,\qquad\forall\,w_n\in L^2(\Om)^M,\quad
\left\{\ba{l}
w_n\to 0\mbox{ strongly in }L^2(\Om)^M
\\ \ecart
\dis \sup_{n\geq 0}F_n(w_n)<\infty,
\ea\right.
\eeq
then $\bar v_n$ is a recovery sequence.
\par
Define the number $p$ by
\beq\label{prho}
p:=\left\{\ba{ll}
\dis 1, & \hbox{if }N=2
\\ \ecart
\dis {2\rho\over 1+\rho}\in\left({2N-2\over N+1},2\right), & \hbox{if }N>2.
\ea\right.
\eeq
Then, we have the following compactness result:
\begin{Thm}\label{thmGc}
Assume that conditions \eqref{ellAal}, \eqref{cvAnLa} hold. Then there exists a subsequence of $n$, still denoted by $n$, such that 
\beq\label{gafnf}
F_n \stackrel{\Gamma}\rightarrow F,
\eeq  
and there exist a symmetric non-negative bilinear operator $\nu:D(F)\to \M(\Om)$, a linear operator 
\beq\label{espsi}
\sigma\mbox{ which maps $D(F)$ into }
\left\{\ba{ll}
\M(\Om)^{M\times N}, & \mbox{if }N=2,\ \Lambda\not\in L^1(\Om)
\\ \ecart
L^{p}(\Om)^{M\times N}, & \mbox{if }N>2\hbox{ or }N=2,\ \Lambda\in L^1(\Om),
\ea\right.
\eeq 
and a tensor-valued function
\beq\label{Arho}
\AA\in\left\{\ba{ll}
\M(\Om)^{(M\times N)^2}, & \mbox{if }N=2,\ \Lambda\not\in L^1(\Om)
\\ \ecart
L^1(\Om)^{(M\times N)^2}, & \mbox{if }N=2,\ \Lambda\in L^1(\Om)
\\ \ecart
L^{\rho}(\Om)^{(M\times N)^2}, & \mbox{if }N>2,
\ea\right.
\eeq
satisfying the following conditions:\par\medskip
\begin{itemize}
\item The operators $\nu$ and $\sigma$ are strongly local in the sense
\beq\label{local}
\left\{\ba{l}\dis u,v\in D(F)
\\ \ecart
\dis Du=Dv\hbox{ a.e. in  }\om\subset\Om,\hbox{ open }
\ea\right.
\;\Rightarrow\;
\left\{\ba{l}\dis\nu(u,u)=\nu(v,v)
\\ \ecart\dis \sigma(u)=\sigma(v)
\ea\right.
\quad\hbox{in }\om.
\eeq
\item The operator $\nu$ satisfies the ellipticity condition
\beq\label{ellnual}
\al\int_\Om\,|Dv|^2\,dx\leq\int_\Om d\nu(u,u),\quad\forall\,u\in D(F).
\eeq
\item The tensor-valued measure $\AA$ satisfy the following bounds
\beq\label{acoA}
\left\{\ba{lll}
|\AA|\leq \Lambda & \hbox{ in }\Om, & \hbox{if }N=2
\\ \ecart
\dis |\AA|\leq(\Lambda^L)^{1\over \rho} & \hbox{ a.e. in }\Om, & \hbox{if }N>2,
\ea\right.
\eeq
where $\Lambda^L$ is the absolute continuous part of $\Lambda$ with respect to Lebesgue's measure.
\item The operators $\nu$, $\sigma$ and the tensor $\AA$ are related by
\begin{itemize}
\item For any $u,v\in D(F)$ and any open set $\om\subset\Om$,
\beq\label{relmusi} 
\left.\ba{ll}
v\in C^1(\om)^M, & \hbox{if }N=2,\ \Lambda\not\in L^1(\Om)
\\ \ecart
v\in W^{1,p'}(\om)^M, & \hbox{if }N>2\hbox{ or }N=2,\ \Lambda\in L^1(\Om) 
\ea\right\}
\Rightarrow\ \nu(u,v)=\sigma(u):Dv\hbox{ in }\om.
\eeq
\item If $N=2$, $\Lambda\not\in L^1(\Om)$, we have for any open set $\om\subset\Om$,
\beq\label{relsiA1}
\sigma(u)=\AA Du\;\;\hbox{in }\om,\quad\forall\,u\in D(F)\cap C^1(\om)^M.
\eeq
\item If $N=2$, $\Lambda\in L^1(\Om)$ or $N>2$, we have
\beq\label{relsiA2}
\sigma(u)=\AA Du\;\;\hbox{a.e. in }\Om,\quad\forall\, u\in W^{1,p'}_0(\Om)^M.
\eeq
\end{itemize}
Moreover, denoting by $\nu^L$ the absolute continuous part of $\nu$ with respect to Lebesgue's measure, we have
\beq\label{relsiA3}
\AA Du:Dv\in L^1(\Om),\quad\nu^L(u,v)=\AA Du:Dv\;\;\hbox{a.e. in }\Om,\qquad\forall\,u,v\in D(F),
\eeq
\item The functional $F$ is given by
\beq\label{funcF}
F(u)=\into d\nu(u,u),\quad\forall\, u\in D(F).
\eeq
\item For any recovery sequence $u_n$ for $F_n$ of limit $u\in D(F)$, we have
\beq\label{cvAnDunDun}
\AA_n Du_n:Du_n\stackrel{\ast}\rightharpoonup\nu(u,u)\quad\hbox{in }\M(\Om),
\eeq
\beq\label{cvAnDun}
\AA_n Du_n\rightharpoonup\sigma(u)\quad
\left\{\ba{ll}
\mbox{weakly $*$ in }\M(\Om)^{M\times N}, & \mbox{if }N=2
\\ \ecart
\mbox{weakly in }L^{p'}(\Om)^{M\times N}, & \mbox{if }N>2.
\ea\right.
\eeq
\end{itemize}
\end{Thm}
\begin{Rem} Assuming $N>2$ or $N=2$, $\Lambda\in L^1(\Om)$.
We deduce from \eqref{relmusi}, \eqref{relsiA2} and \eqref{funcF}, the following integral representation of $F$
\beq\label{repiF}
F(u)=\into \AA Du:Du\, dx,\quad\forall\, u\in W^{1,p'}_0(\Om)^M.
\eeq
If $N=2$, $\Lambda\not\in L^1(\Om)$, the above representation is also true for $u\in D(F)\cap C^1(\Om)$, but in this case the integral must be understood as an integral with respect to the measure $\AA$ and not with respect to Lebesgue's measure.
\end{Rem}
\begin{Rem} Let $f_n$ be a sequence which converges strongly to some $f$ in $H^{-1}(\Om)^M$ and let $u_n$ be the
solution of 
\beq\label{ecufn}
\left\{\ba{l}
\dis -\,{\rm Div}\,(\AA_n Du_n)=f_n\;\;\hbox{in }\Om
\\ \ecart
\dis u_n\in H^1_0(\Om)^M.
\ea\right.
\eeq
By \eqref{ellAal} $F_n(u_n)$ is bounded, and thus, up to a subsequence, there exists $u\in D(F)$ such that $u_n$ converges weakly to $u$ in $H^1_0(\Om)$. Since $F_n$ $\Gamma$-converges to $F$, this implies that $u_n$ is a recovery sequence for $F_n$ and that $u$ is the solution of 
\beq\label{eculi}
\left\{\ba{l}
\dis u\in D(F)
\\ \ecart
\dis \Psi(u,v)=\langle f,v\rangle,\quad\forall\,v\in D(F),
\ea\right.
\eeq
where $\Psi$ the bilinear form associated with $F$. By a uniqueness argument it is not necessary to extract any subsequence.
Moreover, convergence \eqref{cvAnDun} implies that $u$ is a solution of 
\beq\label{eculidi}
-\,{\rm Div}\,\sigma(u)=f\;\;\hbox{in }\Om,
\eeq
which taking into account \eqref{relsiA1}, \eqref{relsiA2}, can be read as
\beq\label{eculidi2}
-\,{\rm Div}\,(\AA Du)=f\;\;\hbox{in }\Om,
\eeq
in the following cases: $N>2$, $N=2$ and $\Lambda\in L^1(\Om)$, $N=2$ and $u\in C^1(\Om)$.
\end{Rem}
\begin{Rem}\label{rem23}
When $N=2$, the boundedness of $\AA_n$ in $L^1(\Om)^{(M\times N)^2}$ ensures the convergence \eqref{cvAnDun} of the flux. Similar compactness results in dimension two were obtained in the conductivity case \cite{BrCa,BrCa1} and in the elasticity case \cite{BrCE1}. When $N>2$, convergence \eqref{cvAnDun} holds when $\AA_n$ is bounded in $L^\rho(\Om)^{M\times N}$ with $\rho>(N-1)/2$. This condition is  stronger than the equi-integrability of $\AA_n$ in $L^1(\Om)^{M\times N}$, which leads to a compactness result in the scalar case of \cite{CaSb} ($M=1$). The proof of the scalar case is based on the maximum principle which does not hold for systems ($M>1$).
\end{Rem}
\noindent
{\bf Proof of Theorem~\ref{thmGc}.}
First all, note that by H\"older's inequality and \eqref{prho} we have
\[
\int_\Om\AA_n Du:Du\,dx\leq
\left\{\ba{ll}
\dis \left(\int_\Om|\AA_n|\,dx\right)\|Du\|_{L^\infty(\Om)^{M\times N}}, & \mbox{if }N=2
\\ \ecart
\dis \left(\int_\Om|\AA_n|^\rho\,dx\right)^{1\over\rho}\left(\int_\Om|Du|^{p'}\,dx\right)^{2\over p'}, & \mbox{if }N>2,
\ea\right.
\quad\forall\,u\in H^1_0(\Om)^M,
\]
which combined with condition \eqref{cvAnLa} implies that the domain of the $\Gamma$-limit satisfies
\beq\label{D(F)}
D(F)\supset \left\{\ba{ll}
\dis W^{1,\infty}_0(\Om)^M, & \mbox{if }N=2
\\ \ecart
\dis W^{1,p'}_0(\Om)^M, & \mbox{if }N>2.
\ea\right.
\eeq
\par
As above mentioned the existence of a subsequence of $n$ and a functional $F$ satisfying \eqref{gafnf} is well known.
The proof is divided in three steps.
\par\ms\noindent
{\it First step: Determination of the operators $\si$ and $\mu$.}
\par\ss\noindent
From \eqref{ellAal} we deduce the inequality
\beq\label{elipF}
\al\int_\Om\,|Dv|^2\,dx\leq F(v),\quad\forall\,v\in D(F),
\eeq
which combined with $C^1_0(\Om)^M\subset D(F)$ shows that $D(F)$ is continuously and densely imbedded in $H^1_0(\Om)^M$, and thus that $H^{-1}(\Om)^M$ is continuously and densely imbedded in $D(F)'$.
\par
Denoting by $\Psi:D(F)\times D(F)\to\RR$ the bilinear form associated with $F$ and taking a countable dense subset $\E$ of $L^2(\Om)^{M}$, define the set $E$ by
\beq\label{defD}
E:=\left\{u\in D(F):\exists\,f\in \E,\;\;\Psi(u,v)=\into f\cdot v\,dx,\ \forall\, v\in D(F)\right\}
\eeq
which is a dense and countable subset of $D(F)$.\par
For $f\in \E$, consider the solution $u_n$ of 
\beq\label{ethhm1}
\left\{\ba{l}\dis -\,\Div\,(\AA_nDu_n)=f\;\;\hbox{in }\Om
\\ \ecart
\dis u_n\in H^1_0(\Om)^M.\ea\right.
\eeq
By \eqref{ellAal} the sequence $u_n$ satisfies the estimate
\beq\label{ethhm0}
\into \AA_nDu_n:Du_n\,dx+\into |Du_n|^2\,dx\leq C.
\eeq
Hence, up to a subsequence, there exist $u\in H^1_0(\Om)^M$ and $\mu_u\in \M(\Om)^M$ such that 
\beq\label{ethhm0a} u_n\rightharpoonup u\;\;\hbox{in }H^1_0(\Om)^M,\eeq
\beq\label{ethhm0b} \AA_nDu_n:Du_n\stackrel{\ast} \rightharpoonup\mu_u \;\;\hbox{in }\M(\Om).\eeq
Taking into account that \eqref{ethhm1}  implies \eqref{recseq}, we  deduce that $u_n$ is a recovery sequence for $F_n$ of limit $u$ and that
\beq\label{ethhm1b}
\Psi(u,v)=\into f\cdot v\,dx,\quad \forall\, v\in D(F),
\eeq
Hence, $u$ is the element of $E$ associated with the function $f\in \E$ and
\beq\label{muOmF}
\mu_u(\Om)=F(u).
\eeq
By H\"older's inequality we have for any $\phi\in C^0_0(\Om)$, $\phi\geq 0$,
\[
\ba{l}\dis\into |\AA_nDu_n|^p\,\phi\,dx\leq \into (\AA_nDu_n:Du_n)^{p\over 2}\,|\AA_n|^{p\over 2}\,\phi\,dx
\\ \ecart
\dis\leq\left(\into \AA_nDu_n:Du_n\,\phi\,dx\right)^{p\over 2}\left(\into |\AA_n|^{p\over 2-p}\,\phi\,dx\right)^{1-{p\over 2}}
\\ \ecart
\dis =\left(\into \AA_nDu_n:Du_n\,\phi\,dx\right)^{p\over 2}\left(\into |\AA_n|^\rho\,\phi\,dx\right)^{1-{p\over 2}}.
\ea
\]
Hence, we deduce the existence of $\sigma_u$ such that
\beq\label{ethhm0d}
\left\{\ba{ll} \AA_nDu_n\stackrel{\ast} \rightharpoonup\sigma_u\;\;\hbox{in }\M(\Om)^{M\times N}, & \hbox{if }N=2
\\ \ecart
\dis \AA_nDu_n \rightharpoonup\sigma_u\;\;\hbox{in }L^p(\Om)^{M\times N},\  |\AA_nDu_n|^p\hbox{ equi-integrable}, & \hbox{if }N>2,
\ea\right.
\eeq
and by \eqref{cvAnLa} for any $\Phi\in C^0_0(\Om)^{M\times N}$,
\beq\label{ethhm0e}
\into \sigma_u:\Phi\,dx\leq
\left\{\ba{ll}
\dis \left(\into |\Phi|\,d\mu_u\right)^{1\over 2}\left(\into |\Phi|\,d\Lambda\right)^{1\over 2}, & \hbox{if }N=2
\\ \ecart
\dis \left(\into  |\Phi|\,d\mu_u\right)^{1\over 2}\left(\into |\Phi|\,d\Lambda\right)^{{1\over p}-{1\over 2}}\left(\into |\Phi|\,dx\right)^{{1\over p'}}
& \hbox{if }N>2.
\ea\right.
\eeq
By \eqref{D(F)} and \eqref{ethhm1} $\sigma_u$ also satisfies
\beq\label{ethhm1a}
\into \sigma_u:Dv\,dx=\Psi(u,v),\quad \forall\, v\in
\left\{\ba{ll}
C^1_0(\Om)^M, & \mbox{if } N=2
\\ \ecart
W^{1,p'}_0(\Om)^M, & \mbox{if } N>2.
\ea\right.
\eeq
\par
Since $\E$ is countable, these subsequences can be chosen independently of $f$. Moreover, taking two elements $f,g\in \E$, and denoting by $u$, $\mu_u$, $\sigma_u$ and by $v$, $\mu_v$, $\sigma_v$ the above defined elements associated with $f$ and $g$ respectively, we have 
\beq\label{estmu-mv}
\ba{l}
\dis\|\mu_u-\mu_v\|_{\M(\Om)}\leq \liminf_{n\to\infty}\into\big|A_nDu_n:Du_n-A_nDv_n:Dv_n\big|\,dx
\\ \ecart
\dis =\liminf_{n\to\infty}\into\big|A_nD(u_n+v_n):D(u_n-v_n)\big|\,dx
\\ \ecart
\dis \leq\lim_{n\to\infty}\left(\into A_nD(u_n+v_n):D(u_n+v_n)\,dx\right)^{1\over 2}\left(\into A_nD(u_n-v_n):D(u_n-v_n)\,dx\right)^{1\over 2}
\\ \ecart
\dis=\|u+v\|_{D(F)}\,\|u-v\|_{D(F)}.
\ea
\eeq
and (in the case $N=2$, $L^p(\Om)^{M\times N}$ must be replaced by $\M(\Om)^{M\times N}$)
\beq\label{estsu-sv}
\ba{ll}
\dis\|\sigma_u-\sigma_v\|^p_{L^p(\Om)^{M\times N}} & \dis \leq \liminf_{n\to\infty}\into\big|\AA_nD(u_n-v_n)\big|^p\,dx
\\ \ecart
& \dis \leq\liminf_{n\to\infty}\left(\into \AA_nD(u_n-v_n):D(u_n-v_n)\,dx\right)^{p\over 2}\left(\into |\AA_n|^\rho\,dx\right)^{1-{p\over 2}}
\\ \ecart
& \dis \leq C\,\|u-v\|^p_{D(F)}.
\ea
\eeq
This estimate allows us to extend by continuity the operators
\[
u\in E\mapsto \mu_u\in \M(\Om)\quad\mbox{and}\quad u\in E\mapsto \sigma_u\in
\left\{\ba{ll}
\M(\Om)^{M\times N}, & \mbox{if } N=2
\\ \ecart
L^p(\Om)^{M\times N}, & \mbox{if } N>2,
\ea\right.
\]
to operators defined in the whole domain $D(F)$, that we denote by $\sigma$ and $\mu$.
It is easy to check that $\mu$ is quadratic and $\sigma$ is linear.
Moreover, by \eqref{estmu-mv} and \eqref{estsu-sv} $\mu$ and $\sigma$ satisfy
\beq\label{ethhm2}
\|\mu(u)-\mu(v)\|_{\M(\Om)}\leq \|u+v\|_{D(F)}\,\|u-v\|_{D(F)},\quad\forall\, u,v\in D(F),
\eeq
\beq\label{ethhm3}
\left\{\ba{ll}
\dis \|\sigma(u)\|_{\M(\Om)^{M\times N}} \leq C\,\|u\|_{D(F)}, & \hbox{if } N= 2
\\ \ecart
\dis \|\sigma(u)\|_{L^p(\Om)^{M\times N}} \leq C\,\|u\|_{D(F)}, & \hbox{if } N>2,
\ea\right.
\eeq
\beq\label{ethhm4}
\into d\mu(u)=F(u),\quad\forall\, u\in D(F),
\eeq
\beq\label{ethhm5}
\into \sigma(u):Dv\,dx=\Psi(u,v),\quad \forall\,v\in
\left\{\ba{ll}
C^1_0(\Om)^M & \hbox{if } N=2
\\ \ecart
W^{1,p}_0(\Om)^M, & \mbox{if } N>2,
\ea\right.
\quad\forall\, u\in D(F).
\eeq
\par
Moreover, observe that for a given recovery sequence $\bar u_n$ of limit $\bar u\in D(F)$, taking $f\in \E$ and $u_n$, $u$ the solutions of \eqref{ethhm1}, \eqref{ethhm1b}, we have for any $\phi\in C^0_0(\Om)$,
\beq\label{ethhm5b}
\ba{l}\dis\left|\,\into\AA_nD\bar u_n:D\bar u_n\,\phi\,dx-\into\phi\,d\mu({\bar u})\,\right|
\leq \left|\,\into\AA_nD\bar u_n:D\bar u_n\,\phi\,dx-\into\AA_nDu_n:Du_n\,\phi\,dx\,\right|
\\ \ecart
\dis+ \left|\,\into\AA_nDu_n:Du_n\,\phi\,dx-\into\phi\,d\mu(u)\,\right|
+\left|\,\into\phi\,d\mu(u)-\into\phi\,d\mu({\bar u})\,\right|.
\ea
\eeq
Using that $\bar u_n+u_n$ and $\bar u_n-u_n$ are recovery sequences of limits $\bar u+u$ and $\bar u-u$ respectively, we also have 
\[
\ba{l}\dis\left|\,\into\AA_nD\bar u_n:D\bar u_n\,\phi\,dx-\into\AA_nDu_n:Du_n\,\phi\,dx\,\right|
\\ \ecart
\dis \leq\|\phi\|_{C^0_0(\Om)}\into\big|\AA_nD(\bar u_n+u_n):D(\bar u_n-u_n)\big|\,dx
\\ \ecart
\dis \leq \|\phi\|_{C^0_0(\Om)}\left(\into \AA_nD(\bar u_n+u_n):D(\bar u_n+u_n)\,dx\right)^{1\over 2}\left(\into \AA_n(D\bar u_n-u_n):D(\bar u_n-u_n)\,dx\right)^{1\over 2}
\\ \ecart
\dis \rightarrow \|\phi\|_{C^0_0(\Om)}\,\|\bar u+ u\|_{D(F)}\,\|\bar u- u\|_{D(F)}.
\ea
\]
Therefore, taking the limsup in \eqref{ethhm5b} and using \eqref{ethhm0b}, \eqref{ethhm2}, we get that
\[
\limsup_{n\to\infty}\dis\left|\,\into\AA_nD\bar u_n:D\bar u_n\,\phi\,dx-\into\phi\,d\mu({\bar u})\,\right|\leq
2\,\|\phi\|_{C^0_0(\Om)}\,\|\bar u+ u\|_{D(F)}\,\|\bar u- u\|_{D(F)},
\]
which by the density of $E$ in $D(F)$ implies that
\[
\lim_{n\to\infty}\into\AA_nD\bar u_n:D\bar u_n\,\phi\,dx=\into\phi\,d\mu({\bar u}),\quad\forall\,\phi\in C^0_0(\Om),
\]
Therefore, \eqref{ethhm0b} holds for any $u\in D(F)$ and any recovery sequence $u_n$ of limit $u$.
Analogously, \eqref{ethhm0d} holds for any $u\in D(F)$ and any recovery sequence $u_n$ of limit $u$.
\par
From the quadratic mapping $u\in D(F)\mapsto \mu(u)\in \M(\Om)$, we can now  construct the associated bilinear operator $\nu:D(F)\times D(F)\to \M(\Om)$, defined by
\beq\label{ethm6}
\nu(u,v):={1\over 4}\,\big(\mu(u+v)-\mu(u-v)\big),\quad \forall\, u,v\in D(F),
\eeq
which satisfies
\beq\label{ethm7}\AA_nDu_n:Dv_n\stackrel{\ast}\rightharpoonup \nu(u,v)\quad\hbox{in }\M(\Om),
\eeq
for any $u,v\in D(F)$ and any recovery sequences $u_n$ and $v_n$  of limits $u$ and $v$ respectively.
\par\ms\noindent
{\it Second step: Use of the div-curl result for the derivation of $\si$ and $\mu$.}
\par\ss\noindent
Let $u_n$ be the solution of equation \eqref{ethhm1} with $f\in E\subset D(F)$, and let $v_n$ be a recovery sequence of limit $v\in D(F)$.
Let us check that the sequences $\sigma_n:=\AA_nDu_n$ and $\eta_n:=Dv_n$ satisfy the assumptions of Theorem~\ref{thmdc}:
\par
First, by convergences \eqref{ethhm0a} and \eqref{ethhm0d} $\sigma_n$ and $\eta_n$ satisfy \eqref{cvvnpwnq} and \eqref{dvnqcwnp},  where $p\in(1,2)$ and $q=2$ are such that
\[
{1\over p}+{1\over 2}=1+{1\over 2\rho}<1+{1\over N-1},
\]
as well as condition \eqref{siLsetaLs'} with $s_n=2$. Next, by the Cauchy-Schwarz inequality combined with the boundedness of $F_n(u_n)$ and $F_n(v_n)$, the sequence $\sigma_n:\eta_n=\AA_nDu_n:Dv_n$ is bounded in $L^1(\Om)$, so that convergence \eqref{cvsinetanmu} holds (see the comment after \eqref{wcW-11}).
\par 
Then, the limit formulation \eqref{concte1} of Theorem~\ref{thmdc} shows that
\beq\label{ethm8b}
\into\psi\,d\nu(u,v)=\into f\cdot v\,\psi\,dx-\into \big(\sigma(u)\nabla\psi)\cdot v\,dx,
\eeq
where $\psi(x):=\varphi(|x-x_0|)$, and $\varphi$ satisfying the conditions \eqref{concte1} or \eqref{concte3} depending if $N>2$ or $N=2$.
\par
In particular, we can take in \eqref{ethm8b} a function $v\in D(F)$ such that for some open set $\om\subset\Om$,
\beq\label{espv}
v\in \left\{\ba{ll}
C^1(\om)^M, & \hbox{if }N=2,\ \Lambda\not\in L^1(\Om)
\\ \ecart
\dis W^{1,p'}(\om)^M, & \hbox{if }N>2\hbox{ or }N=2,\ \Lambda\in L^1(\Om),
\ea\right.
\eeq
a ball $B(x_0,R)\subset\om$, and a function $\varphi\in W^{1,\infty} (0,\infty)$ with ${\rm supp}\,\varphi\subset [0,R]$.
Also using that \eqref{ethhm1b} and \eqref{ethhm1a} imply
\beq\label{ethm9}
-\,\Div\,\sigma(u)=f\quad\hbox{in }\D'(\Om),
\eeq
and that by \eqref{ethhm0e} (which by continuity holds for any $u\in D(F)$) we have $\sigma(u)\in L^1(\Om)^{M\times N}$ if $\Lambda\in L^1(\Om)$,
we get that
\[
\into\psi\,d\nu(u,v)=\into \sigma(u): D(v\psi)\,dx-\into \big(\sigma(u)\nabla\psi)\cdot v\,dx=\into\sigma(u):Dv\,\psi\,dx.
\]
Taking in this equality $u=v_k\in E$ converging to $v$ in $D(F)$, it follows that
\beq\label{ethm9b}\into\psi\,d\mu(v)=\into\sigma(v):Dv\,\psi\,dx,\eeq
for any radial function $\psi$ with respect to some $x_0\in\om$ and with support in $\om$. By \eqref{ethhm0e} we then have
\[
\into \sigma(v):\Phi\,dx\leq\left\{\ba{ll}\dis  \left(  \into\sigma(v):Dv\,|\Phi|\,dx\right)^{1\over 2}\left(\into |\Phi|\,d\Lambda\right)^{1\over 2},
& \hbox{if }N=2
\\ \ecart
\dis \left(\into\sigma(v):Dv\,|\Phi|\,dx\right)^{1\over 2}\left(\into |\Phi|\,d\Lambda\right)^{{1\over p}-{1\over 2}}\left(\into |\Phi|\,dx\right)^{{1\over p'}},
& \hbox{if }N>2,
\ea\right.
\]
for any $\Phi\in C^0_0(\om)^{M\times N}$ radial with respect to some $x_0\in\om$. By the measures derivation theorem this yields 
\beq\label{ethm9c}
\left\{\ba{lll}
\dis\sigma(v)= H_v\Lambda\;\;\hbox{with}\;\;|H_v|\leq \big(H_v:Dv\big)^{1\over 2} & \Lambda\hbox{-a.e. in }\om, & \hbox{if }N=2
\\ \ecart
\dis |\sigma(v)|\leq\big(\sigma(v):Dv\big)^{1\over 2}\,(\Lambda^L)^{{1\over p}-{1\over 2}} & \hbox{a.e. in }\om, & \hbox{if }N>2,
\ea\right.
\eeq
Therefore, for any function $v$ satisfying \eqref{espv}, we obtain the estimates
\beq\label{ethm11}
\left\{\ba{lll}
\dis \sigma(v)=H_v\Lambda\;\;\hbox{with}\;\;|H_v|\leq |Dv| & \Lambda\hbox{-a.e. in }\om, & \hbox{if } N=2
\\ \ecart
\dis |\sigma(v)|\leq|Dv|\,(\Lambda^L)^{{2\over p}-1} & \hbox{a.e. in }\om, & \hbox{if } N>2.
\ea\right.
\eeq
\par\ms\noindent
{\it Third step: Expressions of $\si$ and $\mu$ in terms of the limit tensor $\AA$.}
\par\ss\noindent
Let $\Om_k$, $k\geq 1$, be an exhaustive sequence of open sets in $\Om$ such that
\beq\label{Omk}
\forall\,k\geq 1,\;\;\overline{\Om_k}\subset\Om_{k+1}\quad\mbox{and}\quad\bigcup_{k\geq 1}\Om_k=\Om.
\eeq
We associate with the open sets $\Om_k$ the functions $\phi_k$ satisfying
\beq\label{psik}
\forall\,k\geq 1,\quad\phi_k\in C^1_c(\Om_{k+1})\;\;\mbox{and}\;\;\phi_k\equiv 1\;\;\mbox{in }\Om_k.
\eeq
Then, define the tensor-valued measure $\AA$ by
\beq\label{Afiej}
\AA(f_i\otimes e_j):=\sum_{k=1}^\infty\sigma_{(\phi_kx_jf_i)}1_{\Om_k},\quad\mbox{for }1\leq i\leq M,\ 1\leq j\leq N.
\eeq
Given $\xi\in\RR^{M\times N}$ and applying \eqref{ethm11} to the functions $v-\phi_k\,\xi x$, we have
\[
\left\{\ba{lll}
\dis\sigma(v)-\AA\xi=H_{(v-\phi_k\xi x)}\Lambda\;\;\hbox{with}\;\;|H_{v-\phi_k\xi x}|\leq |Dv-\xi| & \Lambda\hbox{-a.e. in }\om\cap\Om_k, & \hbox{if }N=2
\\ \ecart
\dis |\sigma(v)-\AA\xi|\leq |Dv-\xi|\,(\Lambda^L)^{{2\over p}-1} & \hbox{a.e. in }\om\cap\Om_k, & \hbox{if }N>2.
\ea\right.
\]
Therefore, we get that
\beq\label{ethm12}
\sigma(v)=\AA Dv\;\;\hbox{in }\om,\quad\hbox{for any }v\in D(F)\hbox{ satisfying }\eqref{espv}.
\eeq
By \eqref{ethm9b} we also have
\beq\label{ethm12b}
\mu(v)=\AA Dv:Dv\;\;\hbox{in }\om,\quad\hbox{for any }v\in D(F)\hbox{ satisfying }\eqref{espv}.
\eeq
As a consequence, we obtain that if $u_1,u_2\in D(F)$ are such that there exists an open set $\om\subset\Om$ with $Du_1=Du_2$ in $\om$, then $\sigma(u_1-u_2)=0$ and $\mu(u_1-u_2)=0$ on $\om$.
Hence, from the Cauchy-Schwarz inequality we deduce that
\beq\label{ethm13a}
\sigma(u_1)-\sigma(u_2)=\sigma(u_1-u_2)=0\quad\hbox{in }\om,
\eeq
\beq\label{ethm13}
\ba{ll}
\dis \|\mu(u_1)-\mu(u_2)\|_{\M(\om)} & \dis =\|\nu(u_1+u_2,u_1-u_2)\|_{\M(\om)}
\\ \ecart
& \dis \leq \|\mu(u_1+u_2)\|^{1\over 2}_{\M(\om)}\,\|\mu(u_1-u_2)\|^{1\over 2}_{\M(\om)}=0,
\ea
\eeq
which establishes the local property \eqref{local} of $\sigma$ and $\nu$.
\par
From now on, assume that $N>2$ or $N=2$ and $\Lambda\in L^1(\Om)$, which implies that $\sigma(u)$ belongs to $L^p(\Om)^{M\times N}$ for any $u\in D(F)$.
For a function $u\in E$ and a ball $B(x_0,2R)\subset\Om$, define
\[
\bar u:={1\over |B_{2R}|}\int_{B(x_0,2R)}u\,dx,
\]
\begin{multline*}
U:=\left\{r\in (R,2R): \int_{\partial B(x_0,r)}\hskip-12pt\big(|u-\bar u|^2+r^2|Du|^2\big)\,ds(x)\right.
\\ \ecart
\left.\leq {2\over R}\int_{B(x_0,2R)}\hskip-12pt\big(|u-\bar u|^2+|x-x_0|^2|Du|^2\big)\,dx\right\}.
\end{multline*}
The set $U$ satisfies
\[
\big|(R,2R)\setminus U\big|\leq
{R/2\over\dis\int_{B(x_0,2R)}\hskip-12pt\big(|u-\bar u|^2+|x-x_0|^2|Du|^2\big)\,dx}\int_R^{2R}\hskip-3pt
\int_{\partial B(x_0,r)}\hskip-12pt\big(|u-\bar u|^2+r^2|Du|^2\big)\,ds(x)\,dr\leq {R\over 2},
\]
hence $|U|\geq R/2$.
\par
Next, define
\[
\varphi (r):={1\over |U|}\int_r^{2R} 1_U\,ds,\;\;\mbox{for }r\geq 0,\quad\mbox{and}\quad\psi(x):=\ph(|x-x_0|),\;\;\mbox{for }x\in\Om,
\]
and $v:=u-\bar u\,\phi$, for $\phi\in C^1_c(\Om)$ with $\phi\equiv 1$ in $B(x_0,2R)$.
By the local property \eqref{local} we have $\mu(u)=\nu(u,v)$ in $B(x_0,2R)$.
Putting this in formula \eqref{ethm8b} and noting that $\psi\equiv 1$ in $B(x_0,R)$, we obtain
\[
\dis \big|\mu(u)\big(\bar{B}(x_0,R)\big)\big|\leq\into\psi\,d\mu(u)=\into f\cdot(u-\bar u)\,\psi\,dx+\into\big(\sigma(u)\nabla\psi\big)\cdot(u-\bar u)\,dx.
\]
First, using Poincar\'e-Wirtinger's inequality in $B(x_0,2R)$ and H\"older's inequality in $\partial B(x_0,r)$, second Sobolev's imbedding of $H^1(\partial B(x_0,r))$ into $L^{p'}\big(\partial B(x_0,r)\big)$ (recall that ${1\over p}>{1\over 2}+{1\over N-1}$), third the definition of~$U$, H\"older's inequality in $(R,2R)$ and again Poincar\'e-Wirtinger's inequality in $B(x_0,2R)$, we get that
\[
\ba{l}
\dis \big|\mu(u)\big(\bar{B}(x_0,R)\big)\big|
\\ \ecart
\dis \leq CR\left(\int_{B(x_0,2R)}|f|^2\,dx\right)^{1\over 2}\left(\int_{B(x_0,2R)}|Du|^2\,dx\right)^{1\over 2}
\\ \ecart
\dis \qquad+\,{2\over R}\int_U\left(\int_{\partial B(x_0,r)}|\sigma(u)|^p\,ds(x)\right)^{1\over p}
\left(\int_{\partial B(x_0,r)}|u-\bar u|^{p'}\,ds(x)\right)^{1\over p'}\,dr
\\ \ecart
\dis \leq CR\left(\int_{B(x_0,2R)}|f|^2\,dx\right)^{1\over 2}\left(\int_{B(x_0,2R)}|Du|^2\,dx\right)^{1\over 2}
\\ \ecart
\dis \qquad +\,CR^{(N-1)({1\over p'}-{1\over 2})-1}\int_U\left(\int_{\partial B(x_0,r)}|\sigma(u)|^p\,ds(x)\right)^{1\over p}
\left(\int_{\partial B(x_0,r)}\big(|u-\bar u|^2+r^2|Du|^2\big)\,ds(x)\right)^{1\over 2}\,dr
\\ \ecart
\dis \leq CR\left(\int_{B(x_0,2R)}|f|^2\,dx\right)^{1\over 2}\left(\int_{B(x_0,2R)}|Du|^2\,dx\right)^{1\over 2}
\\ \ecart
\dis \qquad+\,CR^{N({1\over p'}-{1\over 2})}\left(\int_{B(x_0,2R)}|\sigma(u)|^p\,dx\right)^{1\over p}\left(\int_{B(x_0,2R)}|Du|^2\,dx\right)^{1\over 2}.
\ea
\]
Dividing this inequality by $|B(x_0,R)|$ and passing to the limit as $R$ tends to zero, we deduce that
\beq\label{ethm14}
\big|\mu^L(u)\big|\leq C\,|\sigma(u)|\,|Du|\quad\hbox{a.e. in }\Om,
\eeq
where $\mu^L(u)$ denotes the absolute continuous component of $\mu(u)$ with respect to Lebesgue's measure.
On the other hand, also remark that \eqref{ethhm0e} also implies that
\[
|\sigma(u)|\leq \big|\mu^L(u)\big|^{1\over 2}\,(\Lambda^L)^{{1\over p}-{1\over 2}}\quad\hbox{a.e. in }\Om,
\]
which combined to \eqref{ethm14} gives
\[
|\sigma(u)|\leq C\,|Du|\,(\Lambda^L)^{{2\over p}-1}\quad\hbox{a.e. in }\Om.
\]
This inequality is similar to \eqref{ethm11} (which was proved for $v$ smooth), and thus reasoning as for the derivation of \eqref{ethm12}, we get that
\beq\label{ethm15}
\sigma(u)=\AA Du\;\;\hbox{a.e. in }\Om,
\eeq
for any $u\in E$, and then by continuity for any $u\in D(F).$ Returning to \eqref{ethm14} and taking into account the density of $E$ in $D(F)$, we also have
\[
\big|\mu^L(u)\big|\leq C\,|\AA Du|\,|Du|\;\;\hbox{a.e. in }\Om,\quad\forall\, u\in D(F).
\]
Using this inequality with $u$ replaced by $u-\xi x$, $\xi\in\RR^{M\times N}$, and recalling that the mapping $u\mapsto\mu_L$ is quadratic and nonnegative, we obtain
\[
\ba{l}
\dis \big|\mu^L(u)-\AA\xi:\xi\big|\leq\big|\mu^L(u)-\mu^L({\xi x})\big|\leq\big|\mu^L(u+\xi x)\big|^{1\over 2}\big|\mu^L(u-\xi x)\big|^{1\over 2}
\\ \ecart
\dis \leq C\,\big|\AA (Du+\xi)\big|^{1\over 2}\,|Du+\xi|^{1\over 2}\,\big|\AA(Du-\xi)\big|^{1\over 2}\,|Du-\xi|^{1\over 2}
\quad\hbox{ a.e. in }\Om,\ \forall\,\xi\in\RR^N.
\ea
\]
This finally shows that
\beq\label{ethm16}
\mu^L(u)=\AA Du:Du\;\;\hbox{a.e. in }\Om,\quad\forall\, u\in D(F).
\eeq
\subsubsection{A H-convergence approach}
We have the following H-convergence type result. Is is similar to Theorem~\ref{thmGc} but with different assumptions, which allow to treat the case of non-symmetric tensor-valued functions $\AA_n$. 
\begin{Thm}\label{thmHc}
Let $\AA_n$ be a non-negative tensor-valued function in $L^\infty(\Om)^{(M\times N)^2}$ which satisfies conditions \eqref{ellAal} and \eqref{cvAnLa} with $\La\in L^\infty(\Om)$. Also assume that there exists a constant $C>0$ such that
\beq\label{gCSine}
\left(\AA_n\xi:\eta\right)^2\leq C\left(\AA_n\xi:\xi\right)\left(\AA_n\eta:\eta\right)
\quad\mbox{a.e. in }\Om,\ \forall\,\xi,\eta\in\RR^{M\times N}.
\eeq
Then, there exist a subsequence of $n$, still denoted by $n$, and a tensor-valued function $\AA$ in $L^\infty(\Om)^{(M\times N)^2}$ satisfying \eqref{ellAal} and
\beq\label{ACLa}
|\AA|\leq C\,\|\La\|_{L^\infty(\Om)}\quad\mbox{a.e. in }\Om,
\eeq
such that for any $f\in H^{-1}(\Om)^M$, the solution $u_n$ in $H^1_0(\Om)^M$ of the equation
\beq\label{eqAnun}
-\,\Div\left(\AA_n Du_n\right)=f\quad\mbox{in }\Om
\eeq
satisfies the convergences
\beq\label{unAnDun}
u_n\rightharpoonup u\;\;\mbox{weakly in }H^1_0(\Om)^M,\quad
\AA_n{D}u_n\rightharpoonup\AA{D}u\;\;\left\{\ba{ll}
\mbox{in }\M(\Om)^{M\times N}\,*, & \mbox{if }N=2
\\ \ecart
\mbox{in }L^p(\Om)^{M\times N}, & \mbox{if }N>2,
\ea\right.
\eeq
where $u$ is the solution in $H^1_0(\Om)^M$ of the equation
\beq\label{eqAu}
-\,\Div\left(\AA{D}u\right)=f\quad\mbox{in }\Om.
\eeq
\end{Thm}
\begin{Rem}
The extra condition \eqref{gCSine} compensates the fact that $\AA_n$ is not necessarily symmetric. The price to pay with respect to the $\Gamma$-convergence result of Theorem~\ref{thmGc} is that the limit $\Lambda$ of \eqref{cvAnLa} needs to be in $L^\infty(\Om)$.
\end{Rem}
\begin{Rem}
Theorem~\ref{thmHc} is an extension to non equi-bounded coefficients of the classical H-convergence of Murat-Tartar \cite{Mur,Tar}.
In dimension two Theorem~\ref{thmHc} includes the scalar case of \cite{BrCa} and the elasticity case of \cite{BrCE1}. In higher dimension it generalizes the H-convergence of \cite{BCM} thanks to the improvement of the div-curl result.
\end{Rem}
\par
The proof of  Theorem~\ref{thmHc} follows the same scheme as the Murat-Tartar H-convergence \cite{Mur} and some of its extensions \cite{BrCa,BCM,BrCE1}. In particular it is quite similar to the proof of Theorem~5.2 \cite{BCM} restricted to the linear case, using the new div-curl result of Theorem~\ref{thmdc}. So we omit it.
\subsection{Weak continuity of the Jacobian}
Let $\Om$ be a regular bounded open set of $\RR^N$, $N\geq 2$.
It is well known that the distributional determinant defined for $u=(u^1,\cdots,u^N):\Om\to\RR^N$ by (where cof denotes the cofactors matrix)
\beq
{\rm Det}\,(Du):=\sum_{j=1}^N\partial_j\left[u^i\,{\rm cof}\,(Du)_{ij}\right],\quad\mbox{for }1\leq i\leq N,
\eeq
agrees with the determinant $\det\,(Du)$ if $u\in W^{1,N}(\Om)^N$ (see, {\em e.g.}, \cite{Dac}, Lemma 2.7 for further details), but the situation is more delicate if $u$ is less regular.
There has been a lot of works about the distributional determinant, its link with the determinant and its weak continuity; we refer to \cite{Ball, BaMu, Dac, DaMu, Mor, Mul} for various contributions in the topic. In particular, M\"uller showed \cite{Mul} that
\beq\label{Mdet}
{\rm Det}\,(Du)=\det\,(Du),\quad\forall\,u\in W^{1,s}(\Om)^N,\ \forall\,s\geq {N^2\over N+1},
\eeq
whenever ${\rm Det}\,(Du)\in L^1(\Om)$.
In connection with this result, one has (see \cite{Ball,DaMu}, and also \cite{IwSb,FLM} for refinements)
\beq\label{BDet}
u_n\rightharpoonup u\;\;\mbox{in }W^{1,s}(\Om)^N\;\;\Rightarrow\;\;{\rm Det}\,(Du_n)\rightharpoonup{\rm Det}\,(Du)\;\;\mbox{in }\D'(\Om),
\quad\forall\,s>{N^2\over N+1}.
\eeq
Up to our knowledge the most recent result is due to Brezis and Nguyen \cite{BrNg} who proved that for any $s\in [N-1,\infty]$ and for any vector-valued functions $u_n,u$ in $L^\infty(\Om)^N$,
\beq\label{BNdet}
\left.\ba{ll}
u_n\rightharpoonup u & \mbox{in }W^{1,s}(\Om)^N
\\ \ecart
u_n\to u & \mbox{in }L^{s\over s-N+1}(\Om)^N,\;\;\mbox{or}
\\ 
u_n\to u & \mbox{in }B\!M\!O(\Om)^N,\;\;\hbox{if }s=N-1\geq 2
\ea\right\}
\;\Rightarrow\;{\rm Det}\,(Du_n)\rightharpoonup{\rm Det}\,(Du)\;\;\mbox{in }\D'(\Om).
\eeq
\par
In view of the div-curl results of Section~\ref{s.dc}, we will prove a weak continuity result for the Jacobian assuming that ${\rm Det}\,(Du_n)$ converges weakly in $W^{-1,1}(\Om)$ and that $u_n$ converges slightly better than weakly in $W^{1,N-1}(\Om)$.
\begin{Thm}\label{thmdet}
Let $\Om$ be a bounded open set of $\RR^N$, with $N\geq 2$.
Consider a sequence $u_n$ in $W^{1,N}(\Om)^M$ satisfying
\beq\label{thmde1}
{\rm Det}\,(Du_n)\rightharpoonup \mu\;\;\hbox{in }W^{-1,1}(\Om).
\eeq
We have the following alternative:
\begin{itemize}
\item Assume that there exists $s>N-1$ such that
\beq\label{cvunqN}
u_n\rightharpoonup u\quad\mbox{in }W^{1,s}(\Om)^N.
\eeq
If $u\in W^{1,N}(\Om)^N$, then $\mu={\rm Det}\,(Du)$.
\\
Otherwise, $\mu$ is given by the weak formulation
\beq\label{muDet>N-1}
\left\{\ba{l}
\dis \forall\,B(x_0,R)\Subset\Om,\ \forall\,\varphi\in W^{1,\infty}(0,\infty),\mbox{ with }{\rm supp}\,\varphi\subset [0,R],
\\ \ecart
\dis \langle\mu,\psi\rangle=-\into\left(\sum_{j=1}^N{\rm cof}\,(Du)_{1j}\,\partial_j\psi\, u^1\right)dx,\quad\hbox{where }\psi(x):=\varphi(|x-x_0|).
\ea\right.
\eeq
\item Or else, assume that
\beq\label{codjun}
u_n\rightharpoonup u\quad
\left\{\ba{ll}
\hbox{in }W^{1,N-1}(\Om)^N, & \mbox{if }N>2
\\ \ecart
\hbox{in }BV(\Om)^N *, & \mbox{if }N=2,
\ea\right.
\eeq
and that $\nabla u_n^1$ belongs to $L^{N-1,1}(\Om)^N$ with the equi-integrability condition
\beq\label{equN-1J}
\forall\,\ep>0,\,\exists\,\delta>0,\ \|\nabla u^1_n\|_{L^{N-1,1}(E)^{M\times N}}
\leq\ep,\ \forall\,n\in\NN,\,\forall\,E\mbox{ measurable}\subset\Om,\,|E|<\delta.
\eeq
If $u\in W^{1,N}(\Om)^N$, then $\mu={\rm Det}\,(Du)$.
\\
Otherwise, $\mu$ is given by the weak formulation
\beq\label{muDetN-1}
\left\{\ba{l}
\dis \forall\,B(x_0,R)\Subset\Om,\ \forall\,\varphi\in W^{1,\infty}(0,\infty),\mbox{ with }{\rm supp}\,\varphi\subset [0,R],\mbox{ such that }
\\ \ecart\dis
\exists\,U\mbox{closed set of }[0,R],\mbox{ with }u(x_0+ry)\in C^0(U;X^{1,N-1}(S_{N-1}))^M,\ {\rm supp}\,(\ph')\subset U,
\\ \ecart
\dis \langle\mu,\psi\rangle=-\into\left(\sum_{j=1}^N{\rm cof}\,(Du)_{1j}\,\partial_j\psi\, u^1\right)dx,\quad\hbox{where }\psi(x):=\varphi(|x-x_0|),
\ea\right.
\eeq
and $X^{1,N-1}(S_{N-1})$ is the space defined by \eqref{X1N-1}.
\end{itemize}
\end{Thm}
\begin{Rem}\label{remdet}
The first case of Theorem~\ref{thmdet} provides an improvement of the weak continuity \eqref{BNdet} with $s>N-1$, given in \cite{BrNg}. Indeed, if a sequence $u_n$ converges weakly in $W^{1,s}\Om)$ and strongly in $L^{s\over s-N+1}(\Om)^N$, then by the classical weak convergence of the Jacobian (see, {\em e.g.}, \cite{Dac}, Corollary~2.8) ${\rm cof}\,(Du_n)$ converges to ${\rm cof}\,(Du)$ in $L^{s\over N-1}(\Om)^{N\times N}$. Hence, since the exponents ${s\over s-N+1}$ and ${s\over N-1}$ are conjugate, we obtain the weak convergence
\[
\sum_{j=1}^N u_n^1\,{\rm cof}\,(Du_n)_{1j}\rightharpoonup\sum_{j=1}^N u^1\,{\rm cof}\,(Du)_{1j}\quad\mbox{in }L^1(\Om),
\]
which thus implies assumption \eqref{thmde1}. More generally, a sufficient condition to ensure assumption \eqref{thmde1} is that
\[
\sum_{j=1}^N u_n^1\,{\rm cof}\,(Du_n)_{1j}\quad\mbox{is equi-integrable in }L^1(\Om).
\]
\par
The second case of Theorem~\ref{thmdet} proposes an alternative to the delicate weak continuity \eqref{BNdet} with $s:=N-1$, obtained in \cite{BrNg} (Theorem~1), in which the strong convergence of $u_n$ in $L^\infty(\Om)^N$ is replaced by the weak convergence of ${\rm Det}\,(Du_n)$ in $W^{-1,1}(\Om)$ combined with the equi-integrability of $\nabla u_n^1$ in the Lorentz space $L^{N-1,1}(\Om)^N$. There is no link between these two sets of assumptions.
\end{Rem}
\noindent
{\bf Proof of Theorem~\ref{thmdet}.}
First of all and similarly to the proof of Theorem~\ref{thmdc}, the regularity assumption $u\in W^{1,N}(\Om)^N$ implies that the weak formulations \eqref{muDet>N-1} and \eqref{muDetN-1} lead to the equality $\mu={\rm Det}\,(Du)$. It thus remains to treat the general case for $s>N-1$ and $s=N-1$.
\par\ms\noindent
{\it The case: $s>N-1$.}
\par\ss\noindent
In view of the classical weak continuity \eqref{BDet} we may restrict ourselves to the case $s\leq{N^2\over N+1}$.
Define $p:={s\over N-1}$ and $q:=s$. Let us check that the sequences of vector-valued functions $\eta_n:=\nabla u_n^1$ and $\sigma_n$ defined by
\beq\label{sincof}
\sigma^i_n:=\big[{\rm cof}\,(Du_n)\big]_{1i},\quad\mbox{for }1\leq i\leq N.
\eeq
satisfy the assumptions of Theorem~\ref{thmdc}:
\par
First, $\sigma_n$ and $\eta_n$ satisfy condition \eqref{siLsetaLs'} with exponent $N'\in[p,q']$ since
\[
p={s\over N-1}\leq {N^2\over N^2-1}\leq N'\leq {N^2\over N^2-N-1}=\left({N^2\over N+1}\right)'\leq s'=q'.
\]
while $p,q>1$ satisfy the inequality
\[
{1\over p}+{1\over q}={N\over s}<1+{1\over N-1}.
\]
Next, $\si_n$ is divergence free and by the classical weak convergence of the Jacobian (see \cite{Dac}, Corollary~2.8) $\sigma_n$ converges weakly in $L^p(\Om)^N$ (since $p={s\over N-1}$) to the function $\sigma=(\sigma^1,\dots,\sigma^N)$ given by
\beq\label{sicof}
\sigma^i=\big[{\rm cof}\,(Du)\big]_{1i},\quad\mbox{for }1\leq i\leq N.
\eeq
Moreover, $\eta_n$ is curl free and converges weakly to $\nabla u^1$ in $L^q(\Om)^N$.
Therefore, taking into account convergence \eqref{thmde1}, Theorem~\ref{thmdc} through \eqref{concte1} yields the desired limit formulation \eqref{muDet>N-1}.
\par\ms\noindent
{\it The case: $p:=1$ and $q:=N-1$.}
\par\ss\noindent
Let us check that the sequences of vector-valued functions $\sigma_n$ defined by \eqref{sincof} and $\eta_n:=\nabla u_n^1$ satisfy the assumptions of Theorem~\ref{calipi1}:
\par
As in the previous case $\sigma_n$ and $\eta_n$ satisfy condition \eqref{siLsetaLs'} with $s=N'$.
Next, $\sigma_n$ is divergence free in $\Om$, and by the classical weak convergence of the Jacobian (see, \cite{Dac}, Corollary~2.8) $\sigma_n$ converges weakly-$*$ in $\M(\Om)^N$ to the function $\sigma$ defined by \eqref{sicof}.
Moreover, $\eta_n$ is curl free and converges weakly to $\nabla u^1$ in $L^{N-1}(\Om)^N$.
Therefore, taking into account conditions \eqref{thmde1} and \eqref{equN-1J} Theorem~\ref{calipi1} (see also Remark~\ref{rem.dcLN-11}) applies and leads to the limit formulation~\eqref{muDetN-1}, which concludes the proof.
\cqfd
\par\bigskip
The following example shows that Theorem~\ref{thmdet} does not hold if we just assume that $u_n$ converges weakly in $W^{1,N-1}(\Om)$ for $N>2$, or weakly-$\ast$ in $BV(\Om)$ if $N=2$. We refer to \cite{DaMu} (Theorem~1) to an alternative counterexample with the critical space $W^{1,{N^2/(N+1)}}(\Om)$ related to the weak continuity~\eqref{BDet}.
\begin{Exa}\label{cexDet}
Let $N\geq 2$, and let $\Om$ be the cylinder $B'_1\times(0,1)$,  where $B'_1$ is the unit ball of $\RR^{N-1}$. The points of $\Om$ are denoted by $(x',x_N)$. We also use $x'$ to denote a point of $\Om$ whose last coordinate is zero.
\par
Define in cylindrical coordinates the vector-valued function $u_n$ in $\Om$ by
\[
u_n(x):=(1-r)^n\,(nx',x_N),\quad\mbox{for }x\in\Om,\quad\mbox{where }r:=|x'|.
\]
Then, we have
\[
Du'_n=-\,n\,(1-r)^{n-1}\left[n\,{x'\otimes x'\over r}-(1-r)\,I_{N-1}\right],
\]
\[
\nabla u^N_n=(1-r)^{n-1}\left[-\,n\,{x_N\,x'\over r}+(1-r)\,e_N\right].
\]
Therefore, as a consequence of \eqref{paliej} we get that
\begin{multline*}
\into |Du_n|^{N-1}\,dx\leq C\,n^{2N-2}\int_0^1r^{2N-3}(1-r)^{(n-1)(N-1)}\,dr.
\\ \ecart
+C\,n^{N-1}\int_0^1r^{N-2}(1-r)^{(n-1)(N-1)}\,dr+C\leq C.
\end{multline*}
This  combined with the convergence of $u_n$ to zero a.e. in $\Om$, implies that
\[
\left\{\ba{ll}
\dis u_n\stackrel{\ast}\rightharpoonup 0\;\;\hbox{in }BV(\Om)^N & \hbox{if } N=2
\\ \ecart
\dis u_n\rightharpoonup 0\;\;\hbox{in }W^{1,N-1}(\Om)^N & \hbox{if } N>2.
\ea\right.
\]
On the other hand, it is easy to check that
\[
{\rm Det}\,(Du_n)=\det\,(Du_n)=n^{N-1}\,(1-r)^{nN}-n^N\,r\,(1-r)^{nN-1}r.
\]
Hence, again using \eqref{paliej} we conclude that
\[
\ba{ll}
\dis {\rm Det}\,(Du_n)\stackrel{\ast}\rightharpoonup &
\dis |S_{N-2}|\,\lim_{n\to\infty}\left[\,\int_0^1\big(n^{N-1}r^{N-2}(1-r)^{nN}-n^N r^{N-1}(1-r)^{nN-1}\big)\,dr\,\right]\left(\delta_{\{x'=0\}}\otimes 1\right)
\\ \ecart
& \dis =|S_{N-2}|\,{(N-2)!\over N^N}\left(\delta_{\{x'=0\}}\otimes 1\right)\not= 0\quad(\mbox{with }|S_0|:=1).
\ea
\]
\par
Finally, note that
\[
\lim_{n\to\infty}\|u_n\|_{L^\infty(\Om)}=\lim_{n\to\infty}\left[\max_{r\in[0,1]}\big\{n\,r\,(1-r)^n\big\}\right]={1\over e}>0,
\]
which also illustrates the sharpness of the weak continuity result \eqref{BNdet} in \cite{BrNg}.  
\end{Exa}
\end{document}